\documentclass[11pt]{article}
\usepackage{latexsym, amsfonts, amscd, amssymb, verbatim, amsmath,
enumerate}

\newtheorem{theorem}{Theorem}[section]
\newtheorem{proposition}{Proposition}[section]

\newtheorem{lemma}{Lemma}[section]

\newtheorem{example}{Example}[section]

\newcommand{\proofbox}{\hspace{\fill}{$\Box$}}
\newenvironment{proof}{\textbf{Proof}.}{\proofbox}

\def\emptyset{\mbox{{\rm \O}}}
\def\bar{\overline}
\date{}

\numberwithin{equation}{section}

\begin{document}

\title{Regularity of  multipliers and second-order optimality conditions of KKT-type for semilinear parabolic control problems} 
	
\author{H. Khanh\footnote{Department of Optimization and Control Theory,	Institute of Mathematics, Vietnam Academy of Science and Technology,	18 Hoang Quoc Viet road, Hanoi, Vietnam  and Department of Mathematics, FPT University, Hoa Lac Hi-Tech Park, Hanoi, Vietnam; email: khanhh@fe.edu.vn} \  and  B.T. Kien\footnote{Department of Optimization and Control Theory, Institute of Mathematics, Vietnam Academy of Science and Technology,18 Hoang Quoc Viet road, Hanoi, Vietnam; email: btkien@math.ac.vn}}
	
\maketitle
	
\medskip
	
\noindent {\bf Abstract.} {\small A class of  optimal control problems governed by semilinear parabolic equations with mixed constraints and a box constraint for control variable is considered.   We show that if the separation condition is satisfied, then both optimality conditions of KKT-type and regularity of multipliers  are fulfilled. Moreover, we show that if the initial value is good enough and boundary $\partial\Omega$ has a property of positive geometric density, then multipliers and optimal solutions are H\"{o}lder continuous.}
	
\medskip
	
\noindent {\bf  Key words.} Regularity of multipliers, regularity of optimal solutions, Second-order optimality conditions,  semilinear parabolic control, mixed pointwise constraints.
	
\noindent {\bf AMS Subject Classifications.} 49K20, 35J25
	
\vspace{0.4cm}

\section{Introduction}
	
Let $\Omega$ be a bounded domain in $\mathbb{R}^N$ with $N=2$ or $3$ and its  boundary $\Gamma=\partial \Omega$ is of class $C^2$. Let $D=H^2(\Omega)\cap H_0^1(\Omega)$ and $H=L^2(\Omega)$. We consider  the problem of finding control function $u\in L^\infty(Q)$  and the corresponding state $y\in L^\infty(Q)\cap W^{1, 1}_2(0, T; D, H)$ which solve
\begin{align}
    &J(y, u):=\int_Q L(x, t,  y(x, t), u(x, t)) dxdt\to\inf \label{P1}\\
    &{\rm s.t.}\notag\\
    &\frac{\partial y}{\partial t} +Ay + f(y) =u\quad \text{in}\quad Q= \Omega  \times (0, T)\label{P2} \\
    & y(x, t)=0 \quad \text{on}\quad \Sigma=\Gamma\times [0, T], \quad y(x, 0)=y_0(x) \quad \text{on}\quad \Omega \label{P3}\\
    & g(x, t, y(x, t)) +\epsilon u(x,t)\leq 0 \quad \text{a.a.}\quad (x, t)\in Q \label{P4}\\
    & a\leq u(x, t)\leq b  \quad \text{a.a.}\quad (x, t)\in Q   \label{P5}
\end{align} where $y_0\in H^1_0(\Omega)\cap L^\infty(\Omega)$, $\epsilon > 0$, $a, b\in\mathbb{R}$ with $a<b$,  $A$ denotes a second-order elliptic operator of the form
$$
Ay =  - \sum_{i,j = 1}^N D_j\left(a_{ij}(x) D_i y \right),
$$  $L:Q \times {\mathbb R} \times {\mathbb R} \to {\mathbb R}$,  $f: {\mathbb R} \to {\mathbb R}$ and $g: \Omega\times[0, T]\times \mathbb{R} \to \mathbb{R}$ are of class $C^2$. In problem \eqref{P1}-\eqref{P5},  constraint \eqref{P4} may be seen as a regularization of the constraint $g(x, t, y(x, t))\leq 0$ which never satisfies regular conditions. Throughout of the paper, we denote by $\Phi$ the feasible set to problem \eqref{P1}-\eqref{P5}. 

The theory of optimal control governed by parabolic equations has been interested by many mathematicians so far. The mainly research topics in this area consist of the existence of optimal solutions, optimality conditions and numerical methods of computing  optimal solutions. In order to compute  optimal solutions, we need   optimality conditions of KKT type.  Based on first-order optimality conditions and regularity of multipliers and optimal solutions,  we can find  approximate solutions while the second-order sufficient conditions help us to give error estimates between approximate solutions and exact solutions. Therefore, the optimality conditions and regularity multipliers and optimal solutions play an important role in numerical methods such as the finite element method (FEM).  

In the theory of optimal control problems governed by partial differential equations, we often require that the control variables belong to $L^\infty$. This requirement make  the objective functional and constraint mappings  easily to be differentiable with respect to $u$.  However, in this situation, the Lagrange multipliers corresponding to constraints are finitely additive measures rather than  $L^p-$ functions. This motivates us to  study the regularity of Lagrange multipliers, that is to find conditions under which Lagrange multipliers are functions in $L^p$-spaces.

Recently, there have been many papers dealing with optimality conditions and regularity of mutipliers for  parabolic optimal control problems. For this we refer the reader to \cite{Arada-2000},  \cite{Arada-2002-1}, \cite{Arada-2002-2},  \cite{Casas-2000}, \cite{Raymond}, \cite{Rosch1}, \cite{Rosch2} and \cite{Troltzsch} and many references given therein. Among these papers, N. Arada and J.-P. Raymond \cite{Arada-2000}-\cite{Arada-2002-2} and  A.R\"{o}sch and F.Tr\"{o}ltzsch \cite{Rosch2} focused on first-order optimality conditions and regularity of multipliers as well as optimal solutions. As usual, in  order to derive optimality conditions of KKT-type, the authors need the Slater condition or the Robinson constraint qualification. Meanwhile, to obtain the regularity of multipliers, they required that the so-called separation conditions are satisfied. The separation conditions make sure that the Lagrange multipliers corresponding to mixed constraints are countably additive measures and so they can be represented by densities in $L^1$. 

In this paper we continue develop results in \cite{Arada-2000} and \cite{Rosch2} to derive first-and second order optimality conditions of KKT-type and to study regularity of multipliers and optimal solutions to problem \eqref{P1}-\eqref{P5}. We show that, under the separation conditions, both the Robinson constraint qualification and regularity of multipliers are valid. In order to establish such results, we will embed problem \eqref{P1}-\eqref{P5} into a specific mathematical programming problem in Banach spaces and derive optimality conditions under the Robinson constraint qualification. Some techniques from \cite{Rosch2} and  results of Yosida-Hewitt in  \cite{Yosida} are invoked and useful for our proofs. 

The paper is organized as follows. Section 2 is devoted to auxiliary results. The regularity of multipliers and second-order optimality conditions are provided in Section 3.  The higher regularity of multipliers and optimal solutions are pesented in Section 4. We show that if the boundary is good enough, then multipliers and optimal solutions are H\"{o}der continuous.

\section{Auxiliary results}

\subsection{A specific mathematical programming problem}

Let $E, W$  and $Z$ be Banach spaces with the dual spaces $E^*$, $W^*$ and $Z^*$.  Let $M$ be a nonempty and closed subset in $Z$ and $\bar z\in M$. The set
\begin{align*}
T(M, \bar z):&=\left\{h\in X\;|\; \exists t_n\to 0^+, \exists h_n\to h, \bar z+t_nh_n\in M \ \ \forall n\in\mathbb{N}\right\},
\end{align*}
is called  {\em the contingent cone} to $M$ at $\bar z$, respectively. The set 
$$
N(M, \bar z)=\{z^*\in Z^*| \langle z^*, z\rangle\leq 0\quad\forall z\in T(M, \bar z)\}
$$ is called the normal cone to $M$ at $\bar z$.  It is well-known that when $M$ is convex, then
$$
T (M; \bar w)=\overline{\rm cone}(M-\bar w)
$$ and 
$$
N(M, \bar z)=\{z^*| \langle z^*, z-\bar z\rangle \leq 0\quad \forall z\in M\},
$$
where 
$$
{\rm cone }(M-\bar w):=\{\lambda (h-\bar w)\;|\; h\in M, \lambda>0\}.
$$
Let $\psi: Z \to {\mathbb R}$; $F: Z \to E$ and $G: Z \to W$ be given mappings and $K$ be a nonempty closed convex subset in $W$. We now consider the optimization problem
\begin{align}
&\psi(z) \to \inf \label{mP1}\\
&\text{s.t.}\notag\\
&F(z) = 0, \label{mP2}\\
&G(z) \in K.\label{mP3}
\end{align}  By defining  $K_1=\{0\}\times K$ and $G_1=(F, G)$, we can formulate the problem \eqref{mP1}--\eqref{mP3} in the following form
\begin{equation}
\begin{cases}
&\psi(z) \to \inf \\
&\text{s.t.}\notag\\
&G_1(z) \in K_1.
\end{cases}
\end{equation} 
Given a feasible point $z_0$ of  \eqref{mP1}-\eqref{mP3}.  Then  $z_0$ is said to satisfy the Robinson constraint qualification  if
\begin{align}
\label{Robinson.1}
0\in {\rm int} \{G_1(z_0) + DG_1(z_0)Z -K_1\}. 
\end{align}
According to \cite[condition (2.190), page 71]{Bonnans-2000}, if ${\rm int}(K) \ne \emptyset$ then (\ref{Robinson.1}) is equivalent to 
\begin{align}
    \label{Robinson.2}
    \begin{cases}
        (i) \quad  DF(z_0) \  \ {\rm is \ surjective} \\
        (ii) \quad \exists \widetilde z \in {\rm Ker}DF(z_0) \ \ {\rm such \ that} \ \ G(z_0) +  DG(z_0) \widetilde z  \in {\rm int}(K).
    \end{cases}
\end{align}
 
We impose the following assumptions.

\noindent $(A^k_1)$ The mappings $\psi, F$ and $G$ are of class $C^k$ around $z_0$ with $k=1,2$;

\noindent $(A_2)$ $K$ is a closed convex set with nonempty interior;

\noindent $(A_3)$ $DF(z_0)Z$ is a closed subspace.

The Lagrange function associated with the problem is given by the formula
$$
\mathcal{L}(z, \lambda, e^*, w^*)=\lambda\psi(z) +\langle e^*, F(z)\rangle +\langle w^*, G(z)\rangle,
$$ where $z\in Z$, $\lambda\in\mathbb{R}$, $e^*\in E^*$ and $w^*\in W^*$. We say that a triple $(\lambda, e^*, w^*)$ are Lagrange multipliers at $z_0$ if 
\begin{align}
& D\mathcal{L}(z_0, \lambda, e^*, w^*)=\lambda \nabla\psi(z_0)+ DF(z_0)^* e^* + DG(z_0)^* w^*=0,\\
&\lambda\geq 0,\quad w^*\in N(K, G(z_0)).
\end{align} We denote by $\Lambda[z_0]$ the set of Lagrange multipliers $(\lambda, e^*, w^*)$ at $z_0$ and by $\Lambda_1[z_0]$ the set of Lagrange multipliers $(e^*, w^*)$ corresponding to  $\lambda=1$. 

We have the following result on the first-order optimality conditions. Its proof can be found in \cite{Kien-2022}. However, for conveniences of the readers, we provide here a simple proof of the result.  

\begin{proposition} \label{Pro-FirstOptim} Suppose that $z_0$ is a locally optimal solution to  under which $(A^1_1), (A_2)$ and $(A_3)$ are satisfied. Then the set $\Lambda(z_0)$ is nonempty. In addition, if the Robinson constraint qualification \eqref{Robinson.2} is fulfilled, then $\Lambda_1(z_0)$ is nonempty and weakly star compact. 
\end{proposition}
\begin{proof} If $DF(z_0)Z\neq E$, then  there is a point $e_0\in E$ and $e_0\notin DF(\bar z)Z$. Since $DF(z_0)Z$ is a closed subspace,  the separation theorem (see \cite[Theorem 3.4]{Rudin}) implies that there exists a nonzero functional $e^*\in E^*$ which separate $e_0$ and $DF(\bar z)Z$. By a simple argument, we see that $DF(z_0)^* e^*=0$. Putting $\lambda=0, w^*=0$, we get $(0, e^*, 0)\in\Lambda(z_0)$.
		
Let us consider the case $D F(\bar z)Z= E$. In this situation, we define a set $C$ which consists of vectors $(\mu, e, w)\in\mathbb{R}\times E\times W$ such that there exists $z\in Z$ satisfying
\begin{align}
&\nabla \psi(z_0)z<\mu\\
&\nabla F(z_0)z=e \label{C1}\\
&\nabla G(z_0)z-w\in{\rm cone}({\rm int}K-G(z_0)) 
\end{align} It is clear that $C$ is convex. We now show that $C$ is open. In fact, take $(\hat\mu, \hat e,\hat w)\in C$ corresponding to $\hat z$. Choose $\epsilon>0$ such that $ \psi(z_0)\hat z< \hat\mu-\epsilon$. Then there is $\delta>0$ such that $\psi(z_0) z<\hat\mu-\epsilon$ for all $z\in B(\hat z, \delta)$. It follows that  
$$
\psi(z_0) z<\mu \quad \forall z\in B(\hat z, \epsilon) \text{and}\quad |\mu-\hat\mu|<\epsilon.  
$$
Since ${\rm cone}({\rm int} K-G(\bar z))$ is an open convex cone and
$$
DG(z_0)\hat z-\hat w \in {\rm cone}({\rm int}K-G(z_0)),
$$ the continuity of $DG(\bar z)$ implies that there exit  balls  $B(\hat z, \delta)$ and   $B(\hat w, r)$  such that 
$$
DG(z_0)z -w \in {\rm cone}({\rm int} D - G(z_0)) \quad  \forall (z,w)\in B(\hat z, \delta)\times B(\hat w, r).
$$ Since $DF(z_0)$ is surjective, $DF(z_0)[B(\hat z, \delta)]$ is open. Hence, there exists a number $\alpha>0$ such that 
$$
B(\hat e, \alpha)\subset DF(z_0)[B(\hat z, \delta)]. 
$$ Thus for any $(\mu, e)\in  (-\epsilon+\hat \mu, \epsilon+\hat\mu)\times B(\hat e, \alpha)$, there exists $z\in B(\hat  z, \delta)$ such that 
\begin{align*}
&\nabla\psi(z_0)z<\mu\\
&e=DF(z_0)z
\end{align*} such a $z$ satisfies
$$
DG(\bar z)z -w\in {\rm cone}({\rm int} D-G(\bar z))\quad \forall w\in B(\hat w, r).
$$
Thus we have shown that 
$$
(-\epsilon+\hat\mu, \epsilon+\hat\mu)\times B(\hat e, \alpha)\times B(\hat w, r) \subset C.
$$ Hence $C$ is open.  Using Lemma 3.2 in \cite{Kien-2022} for $m=1$ and $d=0$,  we see that $(0, 0, 0)\notin C$. By  the separation theorem (see \cite[Theorem 1, p. 163]{Ioffe-1979},  there exists a nonzero vector 
$(\lambda, e^*,  w^*)\in \mathbb{R}\times E^*\times \times W^*$ such that 
\begin{align}\label{Seperation2}
\lambda\mu + \langle e^*, e\rangle  +\langle w^*, w\rangle\geq 0\quad \forall (\mu, e,  w)\in C. 
\end{align} 
Fix any $z\in  Z$,  $w'\in{\rm cone }({\rm int }K-G(z_0))$ and $r>0$. Put
\begin{align*}
	& \mu=r+ \nabla \psi(z_0)z,\\
        &e=\nabla F(z_0)z,\\
	& w=\nabla G(z_0)z-w'
\end{align*} Then $(\mu, e, w)\in C$. From this and  \eqref{Seperation2}, we get
\begin{align*}
\lambda(\nabla \psi(z_0)z +r)+ \langle e^*, DF(z_0)z\rangle +\langle w^*, DG(z_0)z\rangle -\langle w^*, w'\rangle\geq 0.
\end{align*} If $\lambda<0$ the by letting $r\to+\infty$, we see that the term on the left hand side approach to $-\infty$ which is impossible. Hence we must have $\lambda\geq 0$  By letting $r\to 0$, we get 
\begin{align}\label{Seperation3}
&\lambda(\nabla \psi(z_0)z+\langle e^*, DF(z_0)z\rangle +\langle w^*, DG(z_0)z\rangle \geq 0\quad\forall z\in Z.
\end{align} This implies that 
\begin{align}\label{Separation3}
\lambda\nabla\psi(z_0) + DF(z_0)^* e^* + DG(z_0)^* w^*= 0
\end{align} 
and 
$$
\langle w^*, w'\rangle \leq 0\quad \forall w'\in{\rm cone}({\rm int}K-G(z_0)). 
$$ Hence $w^*\in N(K, G(z_0))$.  The first conclusion of the proposition is proved.  It remains to show that if the Robinson constraint qualification is valid, then $\lambda=1$. Conversely, suppose that $\lambda=0$. Then we have 
$$
DF(z_0)^* e^* + DG(z_0)^* w^*= 0.
$$ Let $\tilde z\in {\rm ker}DF(z_0)$ satisfy \eqref{Robinson.2}. Then we have 
$$
\langle DF(z_0)^* e^*, \tilde z\rangle  + \langle DG(z_0)^* w^*, \tilde z\rangle=0.
$$ Hence 
$$
\langle e^*,DF(z_0)\tilde z \rangle + \langle w^*, DG(z_0) \tilde z\rangle=0.
$$ This implies that 
$$
\langle w^*, DG(z_0) \tilde z\rangle=0,  
$$ where
$$
DG(z_0) \tilde z \in {\rm int}K- G(z_0)={\rm int}(K-G(z_0)). 
$$ Fixing any $w\in W$, we see that, there exists $s>0$ small enough such that 
$$
sw+ DG(z_0) \tilde z \in K-G(z_0).
$$  Since $w^*\in N(K, G(z_0))$, we have 
$$
\langle w^*, sw+ DG(z_0) \tilde z \rangle =s\langle w^*, w\rangle \leq 0. 
$$ Hence $\langle w^*, w\rangle \leq 0$ for all $w\in W$. This implies that $w^*=0$. Since 
$$
 DF(z_0)^* e^* + DG(z_0)^* w^*=0.
$$ we get $ DF(z_0)^* e^*=0$. Since $DF(z_0)$ is surjective, we get $e^*=0$. Consequently, $(\lambda, e^*, w^*)=(0,0,0)$ which is absurd. The boundedness and compactness of $\Lambda_1[z_0]$ follows from Theorem 3.9 in \cite{Bonnans-2000}.
\end{proof}

To deal with the second-order optimality conditions, we need to define the critical cone.  Let us denote by ${\mathcal C}_0[z_0]$  the set of vectors $d \in Z$ satisfying the following conditions:
\begin{itemize}
\item [$(a_1)$] $\nabla \psi(z_0)d \leq 0,$

\item [$(a_2)$] $DF(z_0)d = 0,$

\item [$(a_3)$] $DG(z_0)d  \in {\rm cone}(K-G(z_0))$
\end{itemize}  and by ${\mathcal C}[z_0]$ the closure of ${\mathcal C}_0[z_0]$ in $Z$, which is called the critical cone at $z_0$. Each vector $d'\in \mathcal{C}[z_0]$ is called a critical direction.  

We have the following result on  first and second-order optimality conditions. 

\begin{proposition}\label{Pro-SecondOptim} Suppose that $(A^2_1)-(A_3)$ are satisfied and  the Robinson constraint qualification \eqref{Robinson.2} is valid. Then if 
$z_0$ is a locally optimal solution, then for each $d\in \mathcal{C}_0[z_0]$, there exists $(e^*, w^*)\in \Lambda_1[z_0]$ such that 
\begin{align}
 D^2_z\mathcal{L}(z_0,  e^*, w^*)[d, d]=D^2\psi(z_0) d^2 + e^* D^2F(z_0)d^2 + w^*D^2G(z_0)d^2\geq 0. 
\end{align} 
\end{proposition}
\begin{proof} The proof follows from Theorem 3.2 in \cite{Kien-2022} with $m=1$ and $\Lambda_1[z_0]$ is a weakly star compact set.

\end{proof}

\subsection{ State equation and linearized equation}

Hereafter $H:=L^2(\Omega)$, $V:=H_0^1(\Omega)$ and $D= H^2(\Omega)\cap H_0^1(\Omega)$. The norm and the scalar product in $H$ are denoted by $|\cdot|$ and $(\cdot, \cdot)_H$,  respectively.   It is known that the embeddings 
$$
D\hookrightarrow V\hookrightarrow H
$$ are compact and each space is dense in the following one. 

Given $0<\alpha\leq 1$, we denote by $C^{0, \alpha}(\bar \Omega)$ and $C^{0, \alpha}(\bar Q)$ the space of  H\"{o}lder continuous functions of order $\alpha$,  on $\bar \Omega$ and $\bar Q$, respectively. 

Let $H^{-1}(\Omega)$ be the dual of $H_0^1(\Omega)$. We define the following function spaces
\begin{align*}
& H^1(Q)=W^{1,1}_2(Q)=\{y\in L^2(Q):  \frac{\partial y}{\partial x_i}, \frac{\partial y}{\partial t}\in L^2(Q)\},\\
&V_2(Q)=L^\infty(0, T; H)\cap L^2(0, T; V),\\
&W(0, T)=\{y\in L^2(0, T; H^1_0(\Omega)): y_t\in L^2(0, T; H^{-1}(\Omega))\},\\
& W^{1,1}_2(0, T; V, H)=\{y\in L^2([0, T], V): \frac{\partial y}{\partial t}\in L^2([0, T], H) \},\\
& W^{1,1}_2(0, T; D, H)=\{y\in L^2([0, T], D): \frac{\partial y}{\partial t}\in L^2([0, T], H) \},\\
 &U=L^\infty(Q),\ K_\infty=\{u\in U: u(x, t)\leq 0\ {\rm a.a.}\ (x, t)\in Q\},\\
 &Y = \bigg\{ y \in W^{1,1}_2(0, T; D, H) \cap L^\infty(Q): \quad \frac{\partial y}{\partial t} + Ay \in L^p(0, T; H) \bigg\},
\end{align*} where
\begin{align} \label{Dimension}
    & \frac{4}{4-N}<p<4 \quad \text{whenever}\quad N=2,\\
    & \frac{4}{4-N}<p< \frac{2N}{N-2} \quad \text{whenever}\quad N=3.
\end{align} $Y$ is endowed with the graph norm
\begin{align*}
    \|y\|_Y:= \|y\|_{W^{1, 1}_2(0, T; D, H)} + \|y\|_{L^\infty(Q)} +  \bigg\|\frac{\partial y}{\partial t} + Ay\bigg\|_{L^p(0, T; H)}.
\end{align*}
Under the graph norm $Y$ is a Banach space. Besides, we have the following embedding: 
\begin{align}\label{keyEmbed1}
&Y \hookrightarrow  W^{1,1}_2(0, T; D, H)\hookrightarrow  C(0, T; V),\\
& W(0, T) \hookrightarrow C([0, T], H),
\end{align} where $C([0, T], X)$ stands for the space of continuous mappings $v: [0, 1]\to X$ with $X$ is a Banach space. Moreover, for all $v, w\in W(0, T)$, we have  the following integration by part formula (see Proposition 23.23, p.422 in \cite{Zeidler}) :
\begin{align}\label{IntegralByPart}
(v(t_1), w(t_1))_H -(v(0), w(0))_H =\int_0^{t_1}\langle v_t, w\rangle dt +\int_0^{t_1}\langle w_t, v\rangle dt\quad \text{with}\quad t_1\in [0, T].
\end{align} 
Recall that given  $y_0\in H$ and $u\in L^2(0, T; H)$,  a function $y\in W(0, T)$ is said to be a weak solution of the semilinear parabolic equation \eqref{P2}-\eqref{P3} if 
\begin{align}\label{GeneralSol}
&\int_{\Omega} y(x, t_1)\eta(x, t_1) dx -\int_0^{t_1}\int_\Omega y(x, t)\eta_t(x, t)dxdt +\sum_{i,j=1}^N \int_0^{t_1}\int_\Omega a_{ij} D_i y(x,t)D_j \eta(x, t)dxdt \notag\\
&+\int_0^{t_1}\int_\Omega f(y(x,t))\eta(x,t)dxdt=\int_{\Omega} y_0(x)\eta(x, 0) dx +\int_0^{t_1}u(x, t)\eta(x,t)dxdt
\end{align} for all test functions $\eta\in W^{1,1}_2(0, T; V, H)$ and $t_1\in [0, T]$.   By  standard arguments and formular \eqref{IntegralByPart}, we can show that if $f(y)\in L^2(0, T; H)$ whenever $y\in L^2(0, T; V)$ then $y$ is a weak solution of equations  \eqref{P2}-\eqref{P3} if and only if
\begin{align}
\label{WeakSol}
\begin{cases}
\langle y_t, v\rangle + \sum_{i,j=1}^n \displaystyle \int_\Omega a_{ij}D_iy D_j v dx +(f(y), v)_H =(u, v)_H\quad \forall v\in H_0^1(\Omega)\quad \text{a.a.}\quad t\in [0, T]\\
y(0)=y_0.
\end{cases}
\end{align} 
If a  weak solution $y$ such that $y\in W^{1,1}_2(0, T; D, H)$ and $f(y)\in L^2(0, T; H)$ then we have $y_t +Ay +f(y)\in L^2(0, T; H)$ and
$$
\langle y_t, v\rangle + (Ay, v) +(f(y), v) =(u, v)\quad \forall v\in H_0^1(\Omega).
$$ Since $H_0^1(\Omega)$ is dense in $H=L^2(\Omega)$, we have  
$$
\langle y_t, v\rangle + (Ay, v) +(f(y), v) =(u, v)\quad \forall v\in H.
$$  Hence
$$
y_t + Ay + f(y) =u\quad \text{a.a.}\ t\in [0, T],\ y(0)=y_0.
$$ In this case we say $y$ is a {\it strong solution} of \eqref{P2}-\eqref{P3}. From now on a  solution to \eqref{P2}-\eqref{P3} is understood a strong solution. 

Let us make the following assumptions which are related to the state equation. 

\noindent $(H1)$ Coefficients ${a_{ij}}=a_{ji} \in {L^\infty }\left( \Omega  \right)$ for every $1 \le i,j \le N$, satisfy the uniform ellipticity conditions, i.e.,  there exists a number $\alpha > 0$ such that 
\begin{align}
\alpha {\left| \xi  \right|^2} \le \sum\limits_{i,j = 1}^N {{a_{ij}}\left( x \right){\xi _i}{\xi _j}} \,\,\,\,\,\,\,\,\forall \xi  \in {{\mathbb R}^N}\,\,\,{\rm {for\,\,\,a.a.}}\,\,\,x \in \Omega .
\end{align}

\noindent $(H2)$ The mapping $f:{\mathbb R} \to {\mathbb R}$ is of class $C^2$ such that $f(0)=0$ and $f'(y)\geq 0$. Furthermore,  for each $M>0$, there exists $k_M>0$ such that 
\begin{align}\label{Phi-LocalLip}
&|f'(y)| \le k_M, \\
&|f(y_1)-f(y_2)|+ |f'(y_1)-f'(y_2)|+|f''(y_1)-f''(y_2)|\leq k_M|y_1-y_2|
\end{align} for all $y, y_1, y_2$ satisfying $ |y|, |y_1|, |y_2|\leq M$.

\begin{lemma}\label{Lemma-stateEq} Suppose that $(H1)$ and $(H2)$ are satisfied and $y_0\in L^\infty(\Omega)\cap H_0^1(\Omega)$. Then for each $u\in L^p(0, T; H)$ with $p>\frac{4}{4-N}$, the state equation \eqref{P2}-\eqref{P3} has a unique  solution   $y\in Y$ and there exist positive constants $C_1>0$ and $C_2>0$ such that 
\begin{align}\label{KeyInq0}
    \|y\|_{L^\infty(Q)}\leq C_1 (\|u\|_{L^p([0, T], H)} + \|y_0\|_{L^\infty(\Omega)})
\end{align} and 
\begin{align}\label{KeyInq1}
    \| \frac{\partial y}{\partial t}\|_{L^2([0, T], H)} +\|y\|_{L^2([0, T], D)}\leq C_2\big(\|u\|_{L^2([0, T], H)} +\|y_0\|_{L^\infty(Q)} +\|y_0\|_V\big).
\end{align}  
\end{lemma}
\begin{proof} By \cite[Theorem 2.1]{Casas-2023} and \cite[III, Lemma 2.2]{Casas-2023}, for each $u\in L^p(0, T; H)$ with $p>\frac{4}{4-N}$, the state equation has a unique solution $y\in L^\infty(Q)\cap W^{1, 1}_2(0, T; V, H)$ such that $f(y)\in L^2(0, T; H)$ and inequalities \eqref{KeyInq0} and \eqref{KeyInq1} are fulfilled.  By $(H2)$, we have $f(y)\in L^\infty(Q)$. Let us claim that $y\in W^{1,1}_2(0, T; D, H)$. In fact, consider the linear equation 
$$
z_t + Az = u-f(y),\  z(0)=y_0.
$$ By \cite[Theorem 5, p. 360]{Evan}, this equation has a unique solution $z\in W^{1, 1}_2(0, T; D, H)\cap L^\infty(0, T; V)$. Since
$$
y_t+A y=u-f(y),\ y(0)=y_0,
$$ we get 
$$
(z_t-y_t) +A(z-y)=0,\ (z-y)(0)=0.
$$ Taking the scalar product in $H$ with $z-y$ and using $(H1)$, we have
$$
\frac{1}2\frac{d}{dt}|z-y|^2 +\alpha |\nabla (z-y)|^2 \leq 0.
$$ Hence $\frac{1}2\frac{d}{dt}|z-y|^2 \leq 0$.  Integrating on $[0, t]$ with $t\in [0, T]$, we get 
$$
|z(\cdot, t)-y(\cdot, t)|^2=0.
$$ This implies that $y=z\in W^{1,1}_2(0, T; D, H)$ and so $y\in Y$. 
\end{proof}

Let us consider the linearized equation 
\begin{align}\label{LinearizedEq1}
   y_t + A y + c(x,t)y =u,\quad y(0)=y_0. 
\end{align} 

\begin{lemma}\label{Lemma-LinearizedEq} Suppose that $u\in L^p(0, T;  H)$ with $p>\frac{4}{4-N}$,  $y_0\in L^\infty(\Omega)\cap H^1_0(\Omega)$ and $c\in L^\infty (Q)$. Then equation \eqref{LinearizedEq1} has a unique solution $y\in Y$.
\end{lemma}
\begin{proof} Since  $c\in L^\infty(Q)$, there exists a constant $c_0>0$ such that $|c(x,t)|\leq c_0$ for a.a. $(x, t)\in Q$.  By Theorem 5 in \cite[Chapter 7]{Evan}, equation \eqref{LinearizedEq1} has a unique solution $y\in W^{1,1}_2(0, T; D, H)\cap L^\infty(0, T; H_0^1(\Omega))$. It follows that $c(x,t) y\in L^p(0, T; H)$. Let us consider the equation
\begin{align}\label{LinearizeEq2}
z_t + Az + c_0 z = u + (c_0-c)y,\quad z(0)=y_0. 
\end{align} By \cite[Theorem 5, Chapter 7]{Evan} and Lemma \ref{Lemma-stateEq}, the equation has a unique solution $z\in  Y$.  From \eqref{LinearizedEq1} and \eqref{LinearizeEq2}, we have
\begin{align}
z_t-y_t +A (z-y) + c_0(z-y)=0,\ (z-y)(0)=0. 
\end{align} Taking the scalar product both side with $z-y$ and using $(H1)$, we get
$$
\frac{1}2\frac{d}{dt}|z-y|^2 + \alpha |\nabla (z-y)|^2 + c_0|z-y|^2 \leq 0. 
$$ Integrating on $[0, T]$, we obtain
$$
\sup_{t\in[0, T]}|z-y|^2  + \alpha \|z-y\|^2_{L^2(0, T; D)} + c_0\|z-y\|^2_{L^2(0, T; H)}\leq 0.
$$ This implies that $z=y$ a.a. in $Q$. The conclusion is followed. 
\end{proof}

\section{Existence of regular multipliers and second-order optimality conditions}

Let $\phi: \Omega\times[0, T]\times\mathbb{R}\times\mathbb{R}\to\mathbb{R}$ be a mapping which stands for $L$ and $g$. We impose the following hypotheses. 

\noindent $(H3)$ $\phi$ is a Carath\'{e}odory  function and  for each $(x,t)\in \Omega\times [0, T]$, $\phi(x, t, \cdot, \cdot)$ is of class $C^2$ and satisfies the following property:  for each $M>0$, there exists $l_{\phi,M}>0$ such that 
\begin{align*}
    & |\phi(x, t, y_1, u_1)-\phi(x, t, y_2, u_2)|+|\phi_y(x, t, y_1, u_1)-\phi_y(x, t, y_2, u_2)|\\
    &\quad \quad +|\phi_u(x, t, y_1, u_1)-\phi_u(x, t, y_2, u_2)|+|\phi_{yy}(x, t, y_1, u_1)-\phi_{yy}(x, t, y_2, u_2)|\\
    &\quad \quad +|\phi_{yu}(x, t, y_1, u_1)-\phi_{yu}(x, t, y_2, u_2)|+ |\phi_{uu}(x, t, y_1, u_1)-\phi_{uu}(x, t, y_2, u_2)|\\
    &\quad \quad \leq l_{\phi,M}(|y_1-y_2|+ |u_1-u_2|)
\end{align*}
for all $(x, t, y_i, u_i)\in \Omega\times[0, T]\times \mathbb{R}\times\mathbb{R}$ satisfying $|y_i|, |u_i|\leq M$ with $i=1,2$. Furthermore, we require that the functions $\phi_y(\cdot, \cdot, 0, 0), \phi_u(\cdot, \cdot, 0, 0), \phi_{yy}(\cdot, \cdot, 0, 0), \phi_{yu}(\cdot, \cdot, 0, 0)$ and $\phi_{uu}(\cdot, \cdot, 0, 0)$ belong to $L^\infty(Q)$.

\noindent $(H4)$  $f_y[x, t] \geq -\frac{1}{\epsilon} g_y[x, t]\geq 0$ for a.a. $(x, t)\in Q$ and there exists $\gamma>0$ such that  
    \begin{align}
     a -\bar u(x, t) + \epsilon\bar u(x,t) + g[x, t]  \leq - \gamma \quad {\rm a.a.} \  (x, t) \in Q.
    \end{align}
 Hereafter, given $\bar z=(\bar y, \bar u)\in\Phi$,  the symbol $\phi[x,t]$ stands for $\phi(x, t, \bar y(x, t), \bar u(x, t))$. 

Note that hypothesis $(H3)$ makes sure that $J$ and $g$ are of class $C^2$ on $Y\times U$. Meanwhile, $(H4)$ guarantees that the Robinson constraint qualification is satisfied and the Lagrange multipliers belong to $L^1(\Omega)$.   Let us give an example under which $(H4)$ is satisfied. 

\begin{example}{\rm Let $\epsilon=1$, $a=0$, $f(x,t,y)=y^3 + y$,  $g(x, t, y)= -y -\gamma$ and $y_0(x)\geq 0$  for a.a. $x\in \Omega$. Then we have $f_y(x,t, y)=3y^2 +1$ and $g_y(x, t, y)=-1$. Hence $f_y(x,t, y)\geq - g_y(x, t, y)\geq 0$.  Let $u\in [0, b]$. Then $0\leq u\in L^\infty(Q)$. The maximum principle implies that the solution $y$ of the state equation corresponding to $u$ satisfies the property that $y(x, t)\geq 0$ for a.a. $(x, t)\in Q$. It follows that 
\begin{align}
a- u(x, t) + g(x, t, y) +  u(x, t)  \leq g(x, t, y)=-y(x, t)-\gamma \leq -\gamma. 
\end{align}
 }
\end{example}

Let us define Banach spaces
\begin{align*}
&Z = Y \times U,\ W = W_1 \times W_2,\quad E = L^\infty(Q), \\
&W_1 = L^p(0, T; H),\  W_2 =  L^\infty(\Omega) \cap H_0^1(\Omega) \quad {\rm with} \quad    p > \frac{4}{4 - N} 
\end{align*} and 
\begin{align}
    &K_\infty=\{v\in L^\infty (Q): v(x, t)\leq 0\ {\rm a.a.}\ (x, t)\in Q\}\\
    &K= K_\infty\times K_\infty \times K_\infty.\label{K-set}
\end{align}
Define mappings $F: Y\times U\to W$ and  $G: Y\times U\to E \times E \times E$ by setting
\begin{align}\label{F-mapping}
     F(y, u) = (F_1(y, u), \  F_2(y, u)) = \bigg(\frac{\partial y}{\partial t} + Ay + f(y) - u,\ \   y(0) - y_0 \bigg)
\end{align} and 
\begin{align}\label{G-mapping}   
     G(y, u) = (G_1(y, u), G_2(y, u), G_3(y, u)),
\end{align}
where
\begin{align}
    &G_1(y, u) = u - b, \label{hamG1} \\
    &G_2(y, u) = - u  + a \label{hamG2} \\
    &G_3(y, u) = \epsilon u + g(y). \label{hamG3}
\end{align} 
By definition of space $Y$, if $(y, u) \in Y\times U$ then $\frac{\partial y}{\partial t} + Ay \in L^p(0, T; H),$ $f(y) \in L^\infty(Q) \subset L^p(0, T; H)$ (since $y \in L^\infty(Q)$) and $u \in L^\infty(Q) \subset L^p(0, T; H)$. Hence  
$$
\frac{\partial y}{\partial t} + Ay + f(y) - u \in L^p(0, T; H) = W_1
$$ and $F_1$ is well defined.  Also, since $y \in W^{1,1}_2(0, T; D, H) \cap L^\infty(Q)$, we have $y \in L^2(0, T; H^2 \cap H_0^1(\Omega)) \cap L^\infty(0, T; L^\infty(\Omega))$. Hence  $y(t) \in (H^2 \cap H_0^1(\Omega))\cap L^\infty(\Omega)$ for all $t\in [0, T]$. Consequently, $F_2$ is well defined and so is $F$. 

Then  problem \eqref{P1}-\eqref{P5} can be formulated in the  form
\begin{align}
    &J(y,u) \to {\rm min}, \label{P'1}\\
    &{\rm s.t.} \nonumber \\
    &F(y,u) = 0,\label{P'2}\\
    &G(y,u)\in K.\label{P'3}
\end{align}
Therefore, we can apply Proposition \ref{Pro-SecondOptim}  for the problem \eqref{P1}-\eqref{P5} to derive necessary optimality conditions. 

Let ${\mathcal C}_0 [\bar z]$ be a set of all $z = (y, u) \in Y\times U$ such that the following conditions hold: 
\begin{itemize}
\item [$(c_1)$] $\displaystyle \int_Q (L_y[x, t]y(x, t) + L_u[x, t]u(x, t))dxdt  \le 0$; 

\item [$(c_2)$] $\dfrac{\partial y}{\partial t} + Ay + f'(\bar y)y = u, y(0) = 0;$ 

\item[$(c_3)$] $(u, -u, g_y[\cdot, \cdot]y + \epsilon u) \in {\rm cone}\big[\big(K_\infty -(\bar u-b)\big)\times\big(K_\infty -(a-\bar u)\big)\times\big(K_\infty-(g[\cdot, \cdot]+\epsilon \bar u\big)\big].$
\end{itemize}
We denote by ${\mathcal C}[\bar z]$ the closure of  ${\mathcal C}_0[\bar z]$  in $Y\times U$. The set ${\mathcal C}[\bar z]$ is called a critical cone to problem \eqref{P1}-\eqref{P5} at $\bar z$. Each vector $(y, u)\in\mathcal{C}[\bar z]$ is called a critical direction. 

Firstly, we have

\begin{lemma}\label{Lemma-DF-surjective} For each $\bar z\in\Phi$, the operator $DF(\bar z)$ is surjective.
\end{lemma}
\begin{proof} It is sufficient to show that $F_y(\bar z)$ is bijective. In fact, under assumptions $(H1)$ and $(H2)$, the linear mapping $F_y(\bar z): Y \to W$, defined by 
$$ F_y(\bar z)y = (F_{1y}(\bar z)y, F_{2y}(\bar z)y) = \bigg(\dfrac{\partial y}{\partial t} + Ay + f'(\bar y)y, \quad y(0) \bigg).
$$ Taking  any $v = (u, y_0) \in W$, $u \in L^p(0, T; H)$ and $y_0 \in L^\infty(\Omega) \cap H_0^1(\Omega)$, we consider equation $F_y(\bar z)y = v$. This equation is equivalent to
$$
\dfrac{\partial y}{\partial t} + Ay + f'(\bar y)y = u,  \quad y(0) = y_0. 
$$ Since $u\in L^p(0, T; H)$ with  $p > \frac{4}{4 - N}$, Lemma \ref{Lemma-LinearizedEq} implies that  the above parabolic equation has a unique solution $y \in Y$. Thus $F_y(\bar z)$ is bijective, and so, $DF(\bar z)$ is surjective.
\end{proof}

\begin{lemma}
\label{Lemma-RobinsonCQ}
Suppose that $\bar z\in\Phi$ and $(H4)$ is satisfied.   Then the Robinson constraint qualification is fulfilled at $\bar z$.
\end{lemma}
\begin{proof}  It is easy to see that 
    \begin{align*}
        {\rm int}K_\infty =  \{  v \in L^\infty(Q) \ {\rm and} \ {\rm essup}v < 0 \}
    \end{align*} and 
    $$
    {\rm int} K=({\rm int }K_\infty)\times ({\rm int }K_\infty)\times({\rm int }K_\infty).
    $$
By Lemma \ref{Lemma-DF-surjective}, condition $(i)$ of (\ref{Robinson.2}) is satisfied. It remains to show that the condition $(ii)$ of (\ref{Robinson.2}) is satisfied, that is
\begin{align}
    & \bar u - b + \widetilde u \le - \delta, \label{rrro.1} \\
    & a - \bar u - \widetilde u \le - \delta, \label{rrro.2} \\
    & g[x, t] + \epsilon \bar u +  g_y[x, t] \widetilde y + \epsilon \widetilde u \le - \delta \label{rrro.3}
\end{align}
for some $\delta > 0$ and for some $\widetilde z = (\widetilde y, \widetilde u) \in Y \times U$ satisfying the equation
\begin{align}
\label{rrro.4}
    \frac{\partial \widetilde y}{\partial t} + A \widetilde y + f'(\bar y) \widetilde y = \widetilde u \quad {\rm in} \ Q, \quad \quad \widetilde y = 0 \quad {\rm on} \ \Sigma, \quad \quad \widetilde y(0) = 0 \quad {\rm in} \ \Omega.
\end{align}
To show the existence of $\widetilde z$, we do this as follows.  For $0<\rho <  1$, we define
\begin{align*}
    &\alpha = \alpha(x, t) :=  \frac{1}{\epsilon}g_y[x, t] \le 0   \quad \quad {\rm a.a.}\quad   \  (x, t) \in Q \\
    &w_{\rho} := \bar u - \rho + \alpha z_{\bar u}.
\end{align*}
Here we denote by $z_\zeta$ the solution of equation $\nabla_{(y, u)}F(\bar z)(y, u) = 0$ with $u = \zeta$. Namely, $z_\zeta$ is a solution of the equation
$$
z_t + Az + f'(\bar y) z=\zeta,\  z(0)=0. 
$$
By Lemma \ref{Lemma-LinearizedEq}, the equation 
\begin{align}
    \frac{\partial \xi}{\partial t} + A \xi + [f'(\bar y)  +\alpha]\xi = w_\rho \quad {\rm in} \ Q, \quad \quad \xi = 0 \quad {\rm on} \ \Sigma, \quad \quad \xi(0) = 0 \quad {\rm in} \ \Omega.
\end{align}
has a unique solution $\xi_\rho$.  This implies that 
$$
 \frac{\partial \xi_\rho}{\partial t} + A \xi_\rho + f'(\bar y)\xi_\rho=  w_\rho  -\alpha\xi_\rho.
$$ Therefore, if we set 
\begin{align}
\label{rro.1}
    u_\rho = w_\rho - \alpha \xi_\rho = \bar u - \rho - \alpha(\xi_\rho- z_{\bar u})
\end{align}
then $\xi_\rho = z_{u_\rho}$. Subtracting the equation satisfied by $z_{\bar u}$ from the equation satisfied by $z_{u_\rho}$, we get
\begin{align}
\label{rro.1.1}
\begin{cases}
    \frac{\partial (z_{u_\rho} - z_{\bar u})}{\partial t} + A (z_{u_\rho} - z_{\bar u}) + f'(\bar y) (z_{u_\rho} - z_{\bar u}) = u_\rho - \bar u \quad {\rm in} \ Q, \\
    z_{ u_\rho} - z_{\bar u} = 0 \quad  {\rm on} \ \Sigma, \\
    (z_{u_\rho} - z_{\bar u})(0) = 0 \quad  {\rm in} \ \Omega.
\end{cases}
\end{align}
On the other hand,  from (\ref{rro.1}) we have
\begin{align}
    \label{rro.2}
    u_\rho - \bar u  = - \rho - \alpha (\xi_\rho- z_{\bar u}) =  -\rho - \alpha (z_{u_\rho}- z_{\bar u}).
\end{align}
Inserting (\ref{rro.2}) into (\ref{rro.1.1}), we obtain
\begin{align}
    \label{rro.3}
    \begin{cases}
        \frac{\partial (z_{u_\rho} - z_{\bar u})}{\partial t} + A (z_{u_\rho} - z_{\bar u}) + [f'(\bar y) + \alpha(x, t)] (z_{u_\rho} - z_{\bar u}) = -\rho \quad  {\rm in} \ Q,  \\
        z_{ u_\rho} - z_{\bar u} = 0 \quad  {\rm on} \ \Sigma,  \\
        (z_{u_\rho} - z_{\bar u})(0) = 0 \quad  {\rm in} \ \Omega.
    \end{cases}
\end{align} In this equation, we have $ -\rho <0$ and 
$f'(\bar y) + \alpha(x, t)=f'(\bar y) +\frac{1}{\epsilon} g_y[x,t]\geq 0$ because of $(H4)$. Then the maximum principle (see \cite[Theorem 7.2, p. 188]{Ladyzhenskaya-1988}) implies that 
\begin{align}
    \label{rro.4}
    \xi_\rho = z_{u_\rho} \leq z_{\bar u} \quad \quad {\rm and} \quad \quad \|z_{u_\rho} - z_{\bar u}\|_{L^\infty(Q)} \le C_1\rho.
\end{align}
Combining this and (\ref{rro.2}), we get
\begin{align}
    \label{rro.5}
    \|u_\rho - \bar u\|_{L^\infty(Q)} \le C_2\rho
\end{align} for some constant $C_2>0$. 
Define 
\begin{align*}
    & Q_a = \{(x, t) \in Q: \bar u(x, t) = a \}, \\
    & Q_\rho = \{(x, t) \in Q: a < \bar u(x, t) \le  a +  2C_2\rho \}.
\end{align*}
Obviously, we have $Q_a  \cap Q_\rho = \emptyset$. Set 
\begin{align}
\label{rro.6}
    \widehat u_\rho = (a + \rho)\chi_a + (\bar u + \rho)\chi_\rho +  u_\rho (1 - \chi_a - \chi_\rho),
\end{align}
where $\chi_a$ is a characteristic function of the set $Q_a$ and $\chi_\rho$ is a characteristic function of $Q_\rho$. We now claim that  for $\rho$ small enough, the couple
\begin{align}
\label{rro.7}
    \widetilde z := (z_{\widehat u_\rho} - z_{\bar u}, \quad \widehat u_\rho - \bar u) \in Y \times  U
\end{align}
satisfies (\ref{rrro.1})-(\ref{rrro.4}).  Indeed, the definition of $\widetilde z$ in (\ref{rro.7}) implies that (\ref{rrro.4}) is satisfied. On $Q_a$, we have $\widehat u_\rho = a + \rho$; on $Q_\rho$, $\widehat u_\rho = \bar u + \rho \ge a + \rho$; and on $Q\backslash (Q_a \cup Q_\rho)$, since (\ref{rro.5}), $\widehat u_\rho = u_\rho \ge \bar u - | u_\rho - \bar u| \geq a + 2C_2\rho - C_2\rho = a + C_2\rho$. Hence, we have
\begin{align}
\label{rro.8}
    \widehat u_\rho \ge {\rm min}(a + \rho, a + C_2\rho) = a + \rho{\rm min}(1, C_2) \ \ {\rm a.a.} \ (x, t) \in Q.
\end{align}
$\bullet$ Verification of (\ref{rrro.1}). Firstly, on $Q_a$, we have
\begin{align}
    \bar u - b +  (\widehat u_\rho - \bar u) = \widehat u_\rho - b = a + \rho - b = \rho - (b - a) \le -  \frac{b - a}{2} < 0
\end{align}
for $\rho$ small enough. Secondly, on $Q_\rho$,  we have
\begin{align}
     \bar u - b +  (\widehat u_\rho - \bar u) = \widehat u_\rho - b = \bar u + \rho - b &\le a + 2C_2\rho + \rho - b \nonumber \\
     &= \rho (1 + 2C_2) - (b - a) \le -  \frac{b - a}{2} < 0
\end{align} for $\rho$ small enough. Thirdly, on $Q\backslash (Q_a \cup Q_\rho)$, we have from the definition of $\alpha$ and  (\ref{rro.4}) that 
$$
\alpha (z_{u_\rho} - z_{\bar u}) \geq 0.
$$  Hence 
\begin{align}
    \bar u - b +  (\widehat u_\rho - \bar u) = \widehat u_\rho - b =  u_\rho - b & = (\bar u - b) - \alpha (z_{u_\rho} - z_{\bar u}) - \rho \leq - \rho.
\end{align}
Thus, (\ref{rrro.1}) is satisfied with $\widetilde z$ defined by (\ref{rro.7}) for $\rho$ small enough.

\noindent $\bullet$ Verification of \eqref{rrro.2}. From (\ref{rro.8}) we have 
\begin{align}
    a - \bar u - (\widehat u_\rho - \bar u) = a - \widehat u_\rho  \le a  - [a + \rho{\rm min}(1, C_2)] = - \rho{\rm min}(1, C_2) < 0.
\end{align}
This implies that  (\ref{rrro.2}) is satisfied with $\widetilde z$ defined in (\ref{rro.7}) and for $\rho$ small enough.\\
$\bullet$ Verification of  (\ref{rrro.3}). Use (\ref{rro.4}) and (\ref{rro.5}), we get
\begin{align}
    \| z_{\widehat u_\rho} - z_{\bar u}\|_{L^\infty(Q)} &\le \| z_{\widehat u_\rho} - z_{u_\rho}\|_{L^\infty(Q)} + \| z_{u_\rho} - z_{\bar u}\|_{L^\infty(Q)} \nonumber \\
    &\le C_5\| \widehat u_\rho -  u_\rho\|_{L^\infty(Q)} + C_1 \rho \nonumber \\
    &= C_5\| (\bar u - u_\rho + \rho)(\chi_a + \chi_\rho) \|_{L^\infty(Q)} + C_1 \rho 
    \le C_4 \rho.
\end{align}
On the other hand, on  $Q_a \cup Q_\rho$, we have the following estimations
\begin{align*}
\bar u \le a + 2C_2 \rho,\  \widehat u_\rho - \bar u = \rho,\
g_y[x, t] \le 0,\  - (z_{\widehat u_\rho} - z_{\bar u}) \le \| z_{\widehat u_\rho} - z_{\bar u}\|_{L^\infty(Q)} \le C_4 \rho
\end{align*}
 Combining these facts with $(H4)$,  we obtain the following estimations on $Q_a \cup Q_\rho$:
\begin{align}
    g[x, t] + \epsilon \bar u + g_y[x, t](z_{\widehat u_\rho} - z_{\bar u}) + \epsilon (\widehat u_\rho - \bar u)  &\le  - \gamma + (\bar u - a) - C_4g_y[x, t]\rho + \epsilon \rho  \nonumber \\
    & \le  - \gamma + 2 C_2 \rho - C_4g_y[x, t]\rho  + \epsilon \rho   \nonumber \\
    & \le  - \gamma +  \rho (2 C_2  - C_4g_y[x, t] + \epsilon  ) \le \frac{- \gamma}{2}
\end{align}
 for $\rho$ small enough. It remains  to show that the condition is satisfied on $Q\backslash (Q_a \cup Q_\rho)$. Indeed, since $g_y[\cdot, \cdot]\leq 0$, $\alpha=\frac{1}{\epsilon}g_y[x, t]\leq 0$. Hence 
 $$
 \alpha (z_{u_\rho} - z_{\bar u})=\frac{1}{\epsilon}g_y[x, t](z_{ u_\rho} - z_{\bar u})\geq 0.
 $$ On the other hand 
\begin{align}
    \widehat u_\rho - u_\rho =
    \begin{cases}
        2 \rho + \alpha (z_{u_\rho} - z_{\bar u}) > 0 \quad {\rm on} \  Q_a \cup Q_\rho, \\
        0 \quad {\rm on} \ Q \backslash (Q_a \cup Q_\rho).
    \end{cases}
\end{align} This implies that $\widehat u_\rho- u_\rho\geq 0$. The maximum principle implies that $z_{\widehat u_\rho}\geq z_{ u_\rho}$.  Hence 
$$
g_y[x, t](z_{\widehat u_\rho} - z_{u_\rho})\leq 0.
$$
Combining these facts with (\ref{rro.2}), we have on $Q \backslash (Q_a \cup Q_\rho)$ that
\begin{align*}
    &(g[x, t] + \epsilon \bar u) + g_y[x, t] (z_{\widehat u_\rho} - z_{\bar u}) + \epsilon ( \widehat u_\rho - \bar u)\\
     &=(g[x, t] + \epsilon \bar u) + g_y[x, t] (z_{\widehat u_\rho} - z_{\bar u}) + \epsilon ( u_\rho - \bar u)\\
    &\leq g_y[x, t] (z_{\widehat u_\rho} - z_{\bar u}) + \epsilon (u_\rho - \bar u) \\
    &= g_y[x, t] (z_{\widehat u_\rho}- z_{u_\rho}) + g_y[x,t]( z_{u_\rho}- z_{\bar u}) + \epsilon ( u_\rho - \bar u)\\
    &=g_y[x, t] (z_{\widehat u_\rho}- z_{u_\rho}) + g_y[x,t]( z_{ u_\rho}- z_{\bar u}) -\epsilon \rho -g_y[x,t]( z_{ u_\rho}- z_{\bar u})\\
    &=g_y[x, t] (z_{\widehat u_\rho}- z_{ u_\rho}) -\epsilon \rho\leq -\epsilon \rho.
\end{align*}  Hence (\ref{rrro.3}) is satisfied with $\widetilde z$ defined in (\ref{rro.7}) and for $\rho$ small enough. The Lemma is proved. 
\end{proof}

In the sequel, we will use the following sets:
    \begin{align}
    & Q_a  := \{ (x, t) \in Q: \bar u(x, t) = a \},\label{Qa} \\
    & Q_b := \{ (x, t) \in Q: \bar u(x, t) = b \}, \label{Qb}\\
    &Q_{ab}:=\{(x, t)\in Q: a<\bar u(x, t)< b\},\\
    & Q_0  := \{ (x, t) \in Q: g[x, t] + \epsilon \bar u(x, t) = 0 \}.\label{Q0}
    \end{align}
The following theorem gives  first and second-order necessary optimality conditions and the regularity of multipliers.

\begin{theorem}\label{Theorem 1}  Suppose that $(\bar y, \bar u) \in \Phi$ is a locally optimal solution of the problem \eqref{P1}-\eqref{P5} under which assumptions $(H1)-(H4)$ are satisfied.  Then for each critical direction $d=(y, u)\in\mathcal{C}[\bar z]$, there exist a function $\varphi \in L^\infty(Q)\cap W^{1,1}_2(0, T; D, H)$ and  functions $e, \widehat e \in L^\infty(Q)$ such that the following conditions are fulfilled:

\noindent $(i)$ (the adjoint equation)
\begin{align}
-\frac{\partial \varphi }{\partial t } + A^* \varphi + f'(\bar y)\varphi = - L_y[\cdot, \cdot]  - e g_y[., .], \ \ \ \varphi(\cdot, T) = 0, 
\end{align}
where $A^*$ is the adjoint operator of $A$, which is defined by 
\begin{align*}
    A^* \varphi = - \sum_{i, j = 1}^N D_i (a_{ij}(x)D_j\varphi);
\end{align*}

\noindent $(ii)$ (the stationary condition in $u$)
\begin{align}
L_u[x,t]-\varphi(x,t) + \epsilon e(x,t) + \widehat e(x, t) = 0 \quad {\rm a.a.}\quad (x,t)\in Q;
\end{align}

\noindent $(iii)$ (the complimentary condition) $e(x, t) \ge 0 $ and $e(x,t)(g[x,t] + \epsilon \bar u(x, t)) = 0$  a.a. $(x, t) \in Q$. $\widehat e(x,t)$ has the property that 
\begin{align}
\widehat e(x, t)
\begin{cases}
\leq 0 \quad &{\rm a.a. }\quad (x,t)\in Q_a, \\
\geq 0\quad & {\rm a.a. } \quad (x,t)\in Q_b, \\
 0\quad &{\rm a.a. } \quad (x,t)\in Q_{ab}.
\end{cases}
\end{align}  	

\noindent $(iv)$ (the non-negative second-order condition) 
\begin{align}
    \int_Q(L_{yy}[x, t]y^2 + 2L_{yu}[x, t]yu + L_{uu}[x, t]u^2)dxdt &+ \int_Q e(x, t)g_{yy}[x, t]y^2 dxdt \nonumber \\ 
    &+ \int_Q \varphi f''(\bar y) y^2 dxdt \geq 0 
\end{align}
\end{theorem} 
\begin{proof}  Let $K$, $F$ and $G$ be defined by \eqref{K-set}, \eqref{F-mapping} and \eqref{G-mapping}, respectively.  Then problem \eqref{P1}-\eqref{P5} is formulated in the form of problem \eqref{P'1}-\eqref{P'3}.  From $(H2)$ and $(H3)$ we see that the mappings $J, F$ and $G$ are of class $C^2$ around $\bar z  = (\bar y, \bar u)$.  Here $\nabla_{(y, u)}J(\bar z),  \nabla_{(y, u)}^2J(\bar z), D_{(y, u)}F(\bar z),  D^2_{(y, u)}F(\bar z)$,  $D_{(y, u)}G(\bar z)$ and $D^2_{(y, u)}G(\bar z)$ are given by 
\begin{align}
    &\nabla_{(y, u)}J(\bar z)(y, u) = \int_Q(L_y[x, t]y + L_u[x, t]u)dxdt, \label{D-J}\\
    &\nabla_{(y, u)}^2J(\bar z)(y, u)^2 = \int_Q(L_{yy}[x, t]y^2 + 2L_{yu}[x, t]yu + L_{uu}[x, t]u^2)dxdt, \label{D2-J} \\
    &D_{(y, u)}F(\bar z)(y, u) = (D_{(y, u)}F_1(\bar z)(y, u),\  D_{(y, u)}F_2(\bar z)(y, u)) \label{D-F}\\
    & \ \ \ {\rm where} \ \ \ D_{(y, u)}F_1(\bar z)(y, u) = \dfrac{\partial y}{\partial t} + Ay + f'(\bar y)y - u \ \ \ {\rm and} \ \ \ D_{(y, u)}F_2(\bar z)(y, u) = y(0), \\
    &D^2_{(y, u)}F(\bar z)(y, u)^2 = \bigg(D_{(y, u)}^2F_1(\bar z)(y, u)^2, \ \ D_{(y, u)}^2F_2(\bar z)(y, u)^2\bigg) \label{D2-F} \\
    & \ \ \ {\rm where} \ \ \ D^2_{(y, u)}F_1(\bar z)(y, u)^2 = f''(\bar y)y^2 \ \ \ {\rm and} \ \ \ \nabla_{(y, u)}^2F_2(\bar z)(y, u)^2 = 0,\\
    &D_{(y, u)}G(\bar z)(y, u) =  \big(D_{(y, u)}G_1(\bar z)(y, u), D_{(y, u)}G_2(\bar z)(y, u), D_{(y, u)}G_3(\bar z)(y, u)\big)\notag\\ 
    &=  (u, - u, g_y[x, t]y + \epsilon u), \label{D-G}\\
    &D^2_{(y, u)}G(\bar z)(y, u)^2  =  \bigg(D^2_{(y, u)}G_1(\bar z)(y, u)^2, D^2_{(y, u)}G_2(\bar z)(y, u)^2, D^2_{(y, u)}G_3(\bar z)(y, u)^2\bigg)\notag\\
    &= (0, 0,  g_{yy}[x, t]y^2)\quad \text{for all}\quad (y, u)\in Z.  \label{D2-G}
\end{align}  

 Let $q$ be a conjugate number of $p$ with $p > \frac{4}{4 - N}$, that is, $\frac{1}{p} + \frac{1}{q} = 1$. Since $H = L^2(\Omega)$, $H=H^*$. Hence dual of $L^p(0, T; H)$ is $L^q(0, T; H)$. We recall that $E^* = (L^\infty(Q))^*$ and $W^* = W_1^* \times W_2^*$, where $W_1^* = L^p(0, T; H)^* = L^q(0, T; H) \subset L^1(0, T; H)$ and $W_2^* = \big( L^\infty(\Omega) \cap H_0^1(\Omega) \big)^*$. Then the Lagrange function which is associated with the problem is defined as follows
\begin{align*}
    &\mathcal{L}: Y \times U \times L^q(0, T; H) \times \big( L^\infty(\Omega) \cap H_0^1(\Omega) \big)^* \times \big((L^\infty(Q))^*\big)^3  \to  \mathbb R, \\
    &\mathcal{L}(y, u, \phi_1, \phi_2^*, e_1^*, e_2^*, e_3^*) = \int_QL(x, t, y, u)dxdt + \int_Q\bigg[\phi_1 (\frac{\partial y}{\partial t} + Ay + f(y) -u) \bigg]dxdt \\
    & \hspace{5cm} + \left\langle \phi_2^*, y(0) - y_0 \right\rangle_{W_2^*, W_2}  +  \sum_{i = 1}^3 \langle e_i^*, G_i(y, u)\rangle_{L^\infty(Q)^*, L^\infty(Q)}.
\end{align*}
By Lemma \ref{Lemma-DF-surjective}, $DF(\bar z)$ is surjective. By Lemma \ref{Lemma-RobinsonCQ} the Robinson constraint qualification is valid at $\bar z$. Hence Proposition \ref{Pro-SecondOptim} is applicable.  According to Proposition \ref{Pro-SecondOptim},  for each $d= (y, u) \in \mathcal{C}[\bar z]$, there exist $\phi_1 \in L^q(0, T; H)$, $\phi_2^* \in \big( L^\infty(\Omega) \cap H_0^1(\Omega) \big)^*$ and $e^*=(e_1^*, e_2^*, e_3^*) \in L^\infty(Q)^*\times L^\infty (Q)^* \times L^\infty (Q)^*$  such that the following conditions hold:
\begin{align}
    &\nabla_{(y, u)}{\mathcal L}(\bar z, \phi_1, \phi_2^*, e^*) = 0,  \label{dk1} \\
    &e_1^* \in N(K_\infty, \bar u - b)  \label{dk2.1}, \\
    &e_2^* \in N(K_\infty, - \bar u + a )  \label{dk2.2}, \\
    &e_3^* \in N(K_\infty, g[\cdot, \cdot]+\epsilon \bar u),  \label{dk.3} \\
    &\nabla_{(y, u)}^2{\mathcal L}(\bar z, \phi_1, \phi_2^*, e^*)d^2 \ge 0 \label{dk4} 
    \end{align}
The condition \eqref{dk1} is equivalent to the two following relations 
\begin{align}
    \label{dk1.1}
    \int_Q L_y[x, t]ydxdt &+ \int_Q\bigg[\phi_1 (\frac{\partial y}{\partial t} + Ay + f'(\bar y)y  \bigg]dxdt \nonumber \\
    &+ \left\langle \phi_2^*, y(0) \right\rangle_{W_2^*, W_2} + \left\langle e_3^*, g_y[x, t]y\right\rangle_{L^\infty(Q)^*, L^\infty(Q)} = 0, \quad \forall y \in Y
\end{align}
and
\begin{align}\label{dk1.2}
\int_Q  L_u[x, t]udxdt - \int_Q\phi_1udxdt + \langle e_1^* - e_2^* + \epsilon e_3^*, u\rangle_{L^\infty(Q)^*, L^\infty(Q)} = 0, \quad \forall u \in U.
\end{align} From \eqref{dk2.1}-\eqref{dk.3}, we see that $e^*_i$ are non-negative functional. 
By Theorem 2.3 in \cite{Yosida}, $e^*_i$ with $i=1,2,3$ are finitely additive measures on $Q$. Let us claim that each $e_i^*$  can be represented by  densities in $L^1(Q)$. This can be done as follows. From Theorem 1.23 in \cite{Yosida}, we have representations
\begin{align}
e_i^*= e^*_{ic}+ e^*_{ip},\ e^*_{ic}\geq 0,\  e^*_{ip}\geq 0,
\end{align} where $e^*_{ic}$ is countably additive and $e^*_{ip}$ is purely finitely additive. 
Our goal is to show that $e^*_{ip}= 0$. Then $e_i^*$ is countably additive  which can be represented by densities in $L^1(Q)$.  According to \cite[Theorem 1.22]{Yosida}, there exists a sequence $(Q_n)_{n \in \mathbb N}$ of Lebesgue measurable sets such that  $Q \supset Q_1 \supset Q_2 \supset ... \supset Q_n \supset . \ . \ .$, ${\rm lim}_{n \to +\infty} |Q_n|  = 0$ and $e^*_{1p}(Q_n) =  e^*_{1p}(Q)$  for all  $n \in \mathbb N$. This means, that $|Q_n| = {\rm meas}(Q_n) \to 0$ as $n \to + \infty$, but 
\begin{align}
    \| e^*_{1p} \|_{L^\infty(Q)^*} = \int_Q d e^*_{1p}= \int_{Q_n} d e^*_{1p} = \langle e^*_{1p}, \chi_{Q_n} \rangle_{L^\infty(Q)^*, L^\infty(Q)} \quad \quad {\rm for \ all} \ n \in \mathbb N.
\end{align}
We now define 
\begin{align*}
    & A_1(\delta) = \{ (x, t) \in Q: \bar u (x, t) -b \ge   - \delta \},  \\
    &  A_2(\delta) = \{(x, t) \in Q: a - \bar u (x, t) \ge -\delta \}, \\
    & A_3(\delta) = \{(x, t) \in Q: g[x, t] + \epsilon \bar u(x, t) \ge -\delta \}.
\end{align*}
Then it is easy to check that $A_1(\delta) \cap A_2(\delta) = \emptyset$ for all $\delta \in (0, \frac{b-a}{2})$. Moreover, condition $(H4)$ guarantees that $A_3(\delta) \cap A_2(\delta) = \emptyset$ for $\delta$ small enough. Hence, we can choose $\delta = \delta^* \in (0, \frac{b-a}{2})$ such that $A_1(\delta^*) \cap A_2(\delta^*) = A_3(\delta^*) \cap A_2(\delta^*) = \emptyset$. Also, we have
\begin{align}
    \int_{Q\backslash A_i(\delta^*)} d e_i^* = 0,\quad i=1,2,3
\end{align}
because the support of $e_i^*$ is  a subset of $A_i(\delta^*)$. Next, we define 
\begin{align}
    \widehat u (x, t) =  
    \begin{cases}
        {\rm max}(\frac{1}{\epsilon}, 1)  \quad \quad {\rm on} \ \  A_1(\delta^*) \cup A_3(\delta^*) \\
        -1 \hspace{1.9cm} {\rm on} \ \  A_2(\delta^*) \\
        0 \hspace{2.2cm} {\rm otherwise}.
    \end{cases}
\end{align}
Then the function $v_n := \chi_{Q_n} \widehat u \in L^\infty(Q)$ for $n = 1, 2, ...$. Without loss  of generality, we can assume $Q_n \subset A_1(\delta^*)$. Inserting $u = v_n$ into (\ref{dk1.2}), we get 
\begin{align}
\label{reg.1}
    - \int_{Q_n} L_u[x, t] \widehat u dxdt + \int_{Q_n} \phi_1 \widehat u dxdt = \langle e_1^* +  \epsilon e_3^*, v_n \rangle_{L^\infty(Q)^*, L^\infty(Q)}.
\end{align}
Here we use the fact that ${\rm supp}(e_2^*) \subset Q \backslash A_1(\delta^*)$ which implies that $\langle e_2^*, v_n \rangle_{L^\infty(Q)^*, L^\infty(Q)} = 0$. Moreover, since $e_3^* \ge 0$ and $v_n \ge 0$ on ${\rm supp}(e_3^*)$, $(e_1^*)^c \ge 0$ and $v_n = \chi_{Q_n}{\rm max}(\frac{1}{\epsilon}, 1) \ge \chi_{Q_n}$ on ${\rm supp}(e_1^*)$, deduce from (\ref{reg.1}), we have
\begin{align*}
- \int_{Q_n} L_u[x, t] \widehat u dxdt + \int_{Q_n} \phi_1 \widehat u dxdt &\geq \langle e_1^*, v_n \rangle_{L^\infty(Q)^*, L^\infty(Q)}  \\
&=\int_Q v_n de^*_{1c} +  \int_Q v_n de^*_{1p} \geq \int_{Q} v_n d e^*_{1p}\\
&\geq\int_{Q_n} d e^*_{1p}= \int_Q d e^*_{1p} = \|e^*_{1p}\|_{L^\infty(Q)^*}.
\end{align*}
By letting $n \to + \infty$, we get $0\geq  \|e^*_{1p}\|_{L^\infty(Q)^*}$. Hence $e^*_{1p} = 0.$ By the same arguments, we can show that $e^*_{2p}=0$ and $e^*_{3p}= 0$. Therefore, the measures $e_1^*, e_2^*$ and $e_3^*$ are countably additive. By the Radon-Nikodym theorem (see Theorem 1.47 in \cite{Adams-1975}), $e_i^*$  can be represented by $L^1-$density, that is, there exists $e_i\in L^1(Q)$ such that 
$$
\int_Q v(x,t) d e_i^*=\int_Q e_i(x, t) v(x,t) dxdt\quad \forall v\in L^\infty(Q),\ i=1,2,3. 
$$ For simplification, we can identify $e_i^*$ by $e_i$ with $i=1,2,3$. Then from \eqref{dk2.1}, \eqref{dk2.2} and \eqref{dk.3}, we have 
\begin{align*}
    &\int_Q e_1(x,t)[\eta (x,t)-(\bar u(x,t)-b)] dxdt\leq 0\quad \forall \eta\in K_\infty;\\ 
    &\int_Q e_2(x,t)[\eta (x,t)-(-\bar u(x,t)+a)] dxdt\leq 0\quad \forall \eta\in K_\infty;\\
    &\int_Q e_3(x,t)[\eta (x,t)-(\epsilon\bar u(x,t)+g[x,t])] dxdt\leq 0\quad \forall \eta\in K_\infty.
\end{align*} By \cite[Corollary 4]{Pales}, we have
\begin{align*}
    &e_1(x,t)\in N((-\infty, 0], \bar u(x,t)-b) \quad \text{a.a}\quad (x, t)\in Q;\\
     &e_2(x,t)\in N((-\infty, 0], -\bar u(x,t)+a) \quad \text{a.a}\quad (x, t)\in Q;\\
     &e_3(x,t)\in N((-\infty, 0], g[x,t]+\epsilon\bar u(x,t)) \quad \text{a.a}\quad (x, t)\in Q.
\end{align*} This implies that $e_i(x, t)\geq 0$ with $i=1,2,3$, $e_1(x,t)(\bar u(x,t)-b)=0$, $e_2(x,t)(-\bar u(x,t)+a)=0$ and $e_3(x,t)( g[x,t]+\epsilon\bar u(x,t))=0$ for a.a. $(x, t)\in Q$. 
By setting 
\begin{align}
\label{def.of.e}
    e:= e_3 \quad \quad {\rm and}  \quad \quad \widehat e := e_1 - e_2,
\end{align} 
we see that $e, \widehat e \in L^1(Q)$ and satisfy the complementary condition $(iii)$ of the theorem. Next we derive the assertion $(ii)$ of the theorem. For this we rewrite \eqref{dk1.1} and \eqref{dk1.2} in  the following equivalent forms:
\begin{align}
    \label{dk2.1.a}
    \int_QL_y[x, t]ydxdt + &\int_Q\bigg[\phi_1 (\frac{\partial y}{\partial t} + Ay + f'(\bar y)y)  \bigg]dxdt \nonumber \\
    &+ \langle \phi_2^*, y(0) \rangle_{W_2^*, W_2} + \int_Q eg_y[x, t]ydxdt = 0, \quad \quad {\rm {for \ all}} \quad y \in Y
\end{align}
and
\begin{align}
\label{dk2.2.a}
    \int_Q L_u[x, t]udxdt - \int_Q\phi_1udxdt +  \int_Q (\epsilon e +  \widehat e)udxdt = 0, \quad \quad {\rm {for \ all}} \quad u \in U.
\end{align} The latter implies that 
\begin{align}\label{StationConditionIn-u1}
    L_u[\cdot, \cdot] -\phi_1 + \epsilon e +  \widehat e=0 \quad \text{a.a.}\quad (x,t)\in Q. 
\end{align} Define $\zeta=\phi_1-L_u[\cdot, \cdot]$. Then we have 
$\zeta\in L^q(0, T; H)$ and 
$$
\epsilon e+ \widehat e=\zeta. 
$$ Let us show that $\widehat e\in L^q(0, T; H)$. In fact, 
for $(x, t)\in Q_b$, we have  
$$
0\leq \widehat e(x, t)=\zeta(x, t)-\epsilon e(x, t)\leq \zeta(x, t);
$$ for $(x, t)\in Q_a$, we have from $(H4)$ that 
$$
\epsilon \bar u(x, t) + g[x,t]\leq-\gamma.
$$ The condition $e(x,t)[\epsilon \bar u(x, t) + g[x,t]]=0$ implies that $e(x, t)=0$ on $Q_a$ and so $\widehat e(x,t)=\zeta(x,t)$ on $Q_a$. Note that $\widehat e(x,t)=0$ for $(x,t)\in Q\setminus (Q_a\cup Q_b)$. For each $t\in[0, T]$, we define 
\begin{align*}
&\Omega_{at}=\{x\in\Omega: (x, t)\in Q_a\},\\
&\Omega_{bt}=\{x\in\Omega: (x, t)\in Q_b\}.
\end{align*} Then we have  
\begin{align*}
\int_0^T\big(\int_\Omega \widehat e(x,t)^2dx\big)^{q/2}dt&=\int_0^T\big(\int_{\Omega_{at}}\widehat e(x,t)^2dx + \int_{\Omega_{bt}}\widehat e(x,t)^2dx\big)^{q/2}dt\\
&\leq \int_0^T\big(\int_{\Omega_{at}}\zeta(x,t)^2dx+ \int_{\Omega_{bt}}\zeta(x,t)^2dx\big)^{q/2}dt \\
&\leq \int_0^T\big(\int_\Omega\zeta(x,t)^2dx\big)^{q/2}dt<+\infty
\end{align*} Hence $\widehat e\in L^q(0, T; H)$ and so is $e$. This implies that $e g_y[\cdot, \cdot]\in L^q(0, T; H)$. 

Let us consider the following equation
\begin{align}\label{Eq.Varphi.11}
    -\frac{\partial \varphi }{\partial t } + A^* \varphi + f'(\bar y)\varphi = - L_y[\cdot, \cdot]  - e g_y[\cdot,\cdot], \quad \quad   \varphi(\cdot, T) = 0, 
\end{align} 
where $- L_y[\cdot, \cdot]- e g_y[\cdot,\cdot]\in L^q(0, T; H)$. 
By changing variable $\widetilde \varphi(t)  = \varphi(T-t)$, the equation (\ref{Eq.Varphi.11}) becomes
\begin{align}
    \label{Eq.Varphi.2}
    \frac{\partial \widetilde \varphi (t)}{\partial t} + A^* \widetilde \varphi (t) +  f'(\bar y (x, T - t))  \widetilde \varphi (t) = -L_y[x, T-t] -  e(x, T-t)g_y[x, T-t], \quad     \widetilde \varphi(\cdot, 0) = 0.
\end{align} Let us choose $r>0$ such that 
$$
\frac{1}r +\frac{N}{2q}=1+\frac{N}4. 
$$ Then for $N=2$, we have from \eqref{Dimension} that  $q\in (\frac{4}3, 2)$. Hence $r\in (\frac{3}2, 2)$. When $N=3$, we have from   \eqref{Dimension} that  $q\in (\frac{6}5, \frac{4}3)=(\frac{2N}{N+2}, \frac{4}3)$. Hence $r\in (\frac{8}5, 2)$. It follows that 
$$
L^q(0, T; H)\hookrightarrow L^q(0, T; L^r(\Omega)). 
$$ Hence $- L_y[\cdot, \cdot]- e g_y[\cdot,\cdot]\in L^q(0, T; L^r(\Omega))$. By Theorem 4.1  in \cite{Ladyzhenskaya-1988} (page 153, Chapter 3), the above equation has a unique weak solution  
$\widetilde \varphi\in V_2(Q)= L^\infty(0, T; H)\cap L^2(0, T; V)$.  Hence $\varphi(t) = \widetilde \varphi(T - t)$ is a weak solution to the equation \eqref{Eq.Varphi.11}, that is 
\begin{align*}
&\int_\Omega \varphi(x,t_1)\eta(x,t_1)dx-\int_\Omega\varphi(x, 0)\eta(x,0)dx +\int_Q\varphi \eta_t dxdt +\int_Q(\sum_{i,j=1}^N a_{ij}D_i\eta D_j\varphi +f'(\bar y)\eta\varphi) dxdt\\
&=-\int_Q(L_y[\cdot, \cdot]+eg_y[\cdot, \cdot])\eta dxdt
\end{align*} for all test functions $\eta\in W^{1,1}_2(0, T; V, H)$ and $t_1\in[0, T]$. By taking $\eta=y\in Y$ with $y(0)=0$ and $t_1=T$, we have 
\begin{align*}
    \int_Q\big( y_t \varphi + \sum_{i,j=1}^N a_{ij}D_i y D_j\varphi  + f'(\bar y)y \varphi \big)dxdt = \int_Q\big(-L_y[x, t] - e(x, t)g_y[x, t] \big)ydxdt. 
\end{align*} Using the integration by part formula (see Theorem 1.5.3.1 in \cite{Grisvard}), we get 
\begin{align}\label{Eq.Varphi.4}
    \int_Q\big( y_t + A y + f'(\bar y)y \big) \varphi dxdt = \int_Q\big(-L_y[x, t] - e(x, t)g_y[x, t] \big)ydxdt. 
\end{align}
Subtracting (\ref{dk2.1.a}) from (\ref{Eq.Varphi.4}), we get 
\begin{align}
\label{Eq.Varphi.5}
    \int_Q (\varphi - \phi_1) \bigg( y_t + Ay  + f'(\bar y)y \bigg) dxdt = 0, \quad {\rm for \ all}\ y \in Y, \quad  y(0)=0.
\end{align}
On the other hand, as it was shown that for every $\vartheta \in L^p(0, T; H)$, the parabolic equation
\begin{align*}
    y_t  + Ay  + f'(y)y  = \vartheta, \quad \quad y(0) = 0 
\end{align*}
has a unique solution $y \in Y$. From this and \eqref{Eq.Varphi.5}, we deduce that 
\begin{align*}
\int_Q (\varphi - \phi_1) \vartheta dxdt  =  0 \quad \text {for  all}\quad  \vartheta \in L^p(0, T; H).
\end{align*} This implies  that 
\begin{align*}
  \phi_1= \varphi \in L^\infty (0, T; H)\cap L^2(0, T; V). 
\end{align*} Then \eqref{StationConditionIn-u1} becomes
\begin{align}\label{StationCond-u2}
     L_u[x, t]- \varphi +\epsilon e + \widehat e = 0, \quad \text{a.a.} \quad (x,t) \in Q.
\end{align} Hence assertion $(iii)$ of the theorem is derived. Repeating the procedure in the proof of claim $e, \widehat e\in L^p(0, T; H)$ and using the fact $ L_u[x, t]- \varphi\in L^\infty(0, T; H)$ and equality \eqref{StationCond-u2},  we can show that $e, \widehat e\in L^\infty(0, T; H).$ Hence $-L_y[\cdot, \cdot] -g_y[\cdot, \cdot]e\in L^\infty(0, T, H)$.  By the Lemma \ref{Lemma-stateEq}, the solution $\varphi $ of the adjoint equation belongs to $L^\infty(Q)\cap W^{1,1}_2(0, T; D, H)$. Again, we obtain from \eqref{StationCond-u2} that $e, \widehat e \in L^\infty(Q)$.   Hence assertion $(i)$ of the theorem is established. Asertion $(iv)$ of the theorem follows from condition \eqref{dk4} and formulae \eqref{D2-J}, \eqref{D2-F} and \eqref{D2-G}. The proof of the theorem is complete. 
\end{proof}

\medskip

\medskip

To deal with second-order sufficient optimality conditions, we need to enlarge the critical cone $\mathcal C [\bar z]$ by the cone $\mathcal{C}_2[\bar z]$ which consists  of couples $(y, u) \in W^{1, 1}_2(0, T; D, H) \times L^2(Q)$ satisfying the following conditions:
\begin{itemize}
\item [$(c'_1)$] $\displaystyle \int_Q (L_y[x, t]y(x, t) + L_u[x, t]u(x, t))dxdt \le 0$; 

\item [$(c'_2)$] $\dfrac{\partial y}{\partial t} + Ay + f'(\bar y)y = u, y(0) = 0;$ 

\item[$(c'_3)$] $u(x, t)\in T([a, b], \bar u(x, t))$ for a.a. $(x, t)\in Q$; 

\item[$(c'_4)$] $g_y[x, t]y(x, t) + \epsilon u(x, t))\in T((-\infty, 0], g[x, t]+\epsilon \bar u(x, t))$ for a.a. $(x, t) \in Q$.
\end{itemize}

It easy to see that conditions $(c'_3)$ and $(c'_4)$ can be written in the following form:
\begin{align*}
(c_3^{''})  \quad \quad  &(u(x, t), -u(x, t), g_y[x, t]y(x, t) + \epsilon u(x, t)) \\
&\in T((-\infty, 0], \bar u(x, t)-b) \times T((-\infty, 0], a-\bar u(x, t)) \times T((-\infty, 0], g[x, t]+\epsilon \bar u(x, t))
\end{align*} 
 for a.a. $(x, t) \in Q$.  Hence $\mathcal{C}_2[\bar z]$ contains  the closure of $\mathcal{C}_0[\bar z]$ in $W^{1,1}_2(0, T, D, H)\times L^2(0, T; H)$. 
\medskip

In next part,  we use the Lagrange function of the form 
\begin{align}
    \label{Lagrange.function.2}
    \mathcal L (y, u, \varphi, e, \widehat e) &=  \int_Q L(x, t, y, u)  dxdt  +  \int_Q \bigg[\varphi (\frac{\partial y}{\partial t} + Ay + f(y) - u) \bigg] dxdt  \nonumber \\  
    & \quad \quad  +  \int_{\Omega} \varphi (0)[y(0) - y_0] dx  +  \int_Q e[g(y) + \epsilon u] dxdt  +  \int_Q \widehat e u dxdt.
\end{align}
The following theorem gives second-order sufficient conditions for locally optimal solutions to problem \eqref{P1}--\eqref{P5}. 

\begin{theorem} \label{Theorem 2} Suppose that $(H1), (H2), (H3)$ and $(H4)$ are valid, there exist the multipliers $\varphi \in W^{1, 1}_2(0, T; D, H) \cap L^\infty(Q)$ and $e, \widehat e \in L^\infty(Q)$ corresponding to $(\bar y, \bar u)\in\Phi$ satisfies conclusions $(i), (ii)$ and $(iii)$ of Theorem \ref{Theorem 1} and the following strictly second-order condition:
 \begin{align}\label{StrictSOSCond}
\nabla^2_{(y,u)}\mathcal{L}(\bar y, \bar u, \varphi, e, \widehat e)[(y, v), (y, v)]>0\quad \forall (y, v)\in\mathcal{C}_2[(\bar y, \bar u)]\setminus\{(0,0)\}.
\end{align}  
Furthermore, there exists a number $\Lambda > 0$ such that 
\begin{align}
    L_{uu}[x, t] \ge \Lambda \quad {\rm a.e.} \ (x, t) \in Q.
\end{align}
Then there exist number $\epsilon>0$ and $\kappa > 0$ such that 
\begin{align}
    J(y, u)\geq J(\bar y, \bar u) +\kappa \|u-\bar u\|_{L^2(Q)}^2\quad \forall (y, u)\in\Phi\cap[B_Y(\bar y, \epsilon)\times B_U(\bar u, \epsilon)]. 
\end{align}
\end{theorem}
\begin{proof} Suppose to the contrary that the conclusion were
false. Then, we could find sequences $\{(y_n, u_n)\}\subset \Phi$ and $\{\gamma_n\}\subset \mathbb{R}_+$ such that $y_n\to \bar y$ in $Y$, $u_n\to \bar u$ in $U$,  $\gamma_n\to 0^+$ and
\begin{equation}\label{Iq1}
J(y_n, u_n)< J(\bar y, \bar u) +\gamma_n\|u_n-\bar u\|_{L^2(Q)}^2. 
\end{equation} If $u_n=\bar u$ then by the uniqueness we have $y_n=\bar y$. This leads to  $J(\bar y, \bar u)< J(\bar y, \bar u) + 0$ which is absurd. Therefore, we can assume that $u_n\ne \bar u$ for  all $n\geq 1$.

Define $t_n=\|u_n- \bar u\|_{L^2(Q)}$, $z_n= \frac{y_n-\bar y}{t_n}$ and $v_n= \frac{u_n-\bar u}{t_n}$.
Then $t_n\to 0^+$ and $\|v_n \|_{L^2(Q)}=1$. By the reflexivity of  $L^2(Q)$, we may assume that $v_n \rightharpoonup v$ in $L^2(Q)$. From the above, we have
\begin{equation}\label{KeyIq2}
J(y_n, u_n)- J(\bar y, \bar u)\leq t_n^2 \gamma_n \leq o(t_n^2).
\end{equation} We want to show that $z_n$ converges weakly to some
$z$ in $W^{1, 1}_2(0, T; D, H)$. Since $(y_n,
u_n)\in\Phi$ and $(\bar y, \bar u)\in\Phi$, we have
\begin{align*}
&\frac{\partial y_n}{\partial t}+ Ay_n +f(y_n) = u_n,\ y_n(0)=y_0\\
&\frac{\partial \bar y}{\partial t} +A\bar y + f(\bar y)=\bar u,\ \bar y_n(0)= y_0.
\end{align*} This implies that
\begin{align}\label{Zn}
     \frac{\partial (y_n-\bar y)}{\partial t}+ A(y_n-\bar y) +f(y_n)-f(\bar y) = u_n-\bar u,\ (y_n-\bar y)(0)=0
\end{align}
 By Taylor's expansion, we have 
$$
f(y_n)-f(\bar y)=f'(\bar y +\theta(y_n-\bar y)) (y_n-\bar y),\ 0\leq\theta\leq 1. 
$$ Note that $y_n\to \bar y$ in $L^\infty(Q)$ so there exists $M>0$ such that $\|\bar y +\theta(y_n-\bar y)\|_{L^\infty(Q)}\leq M$. By $(H2)$, there exists $k_M>0$ such that 
$$
|f'(\bar y +\theta(y_n-\bar y))|\leq k_M|\bar y +\theta(y_n-\bar y))| +|f'(0)|. 
$$ Hence $f'(\bar y +\theta(y_n-\bar y))\in L^\infty(Q)$. From \eqref{Zn}, we get
\begin{align}\label{Zn2}
     \frac{\partial z_n}{\partial t}+ Az_n +f'(\bar y+\theta (y_n-\bar y))z_n = v_n, \  z_n(0)=0.
\end{align} By \cite[Theorem 5, p.360]{Evan}, there exists a constant $C>0$ such that 
$$
\|z_n\|_{W^{1,1}_2(0,T; D, H)}\leq  C\|v_n\|_{L^2(Q)}=C.
$$ Hence $\{z_n\}$ is bounded in $W^{1, 1}_2(0,T; D, H)$. Without loss of generality, we may assume that $z_n\rightharpoonup z$ in $W^{1, 1}_2(0,T, D, H)$. By Aubin's lemma, the embedding $W^{1, 1}_2(0, T; D, H)\hookrightarrow L^2([0, T], V)$ is compact. Therefore, we may assume that $z_n\to z$ in $L^2([0, T], V)$. Particularly, $z_n\to z$ in $L^2(Q)$. We now  use the procedure in the proof of \cite[Theorem 3, p. 356]{Evan}. By  passing to the limit,   we obtain from  \eqref{Zn2} that
\begin{align}\label{Z-Eq-C2}
     \frac{\partial z}{\partial t}+ Az +f'(\bar y)z = v,\ z(0)=0.
\end{align}
 Next we claim that $(z, v)\in\mathcal{C}_2[(\bar y, \bar u)]$. Indeed, by a Taylor expansion, we have from \eqref{KeyIq2} that
\begin{equation}\label{CriticalCon1}
J_y(\bar y, \bar u)z_n + J_u(\bar y, \bar u)v_n +\frac{o(t_n)}{t_n}\leq 
\frac{o(t^2_n)}{t_n}.
\end{equation} Note that $L_{u}[\cdot,\cdot]\in L^\infty(Q)$  and
$J_u(\bar y, \bar u) \colon L^2(Q)\to\mathbb{R}$ is a continuous linear mapping, where
$$
\langle J_u(\bar y, \bar u), u\rangle:=\int_Q L_u[x,t] u(x,t)dx dt\ \quad
\forall u\in L^2(Q).
$$ By \cite[Theorem 3.10]{Brezis1}, $J_u(\bar y, \bar u)$ is weakly continuous on $L^2(Q)$. On the other hand, $L_y[\cdot, \cdot]\in L^\infty (Q)$. Hence 
$\lim_{n\to \infty} \int_Q L_y[x,t] z(x,t)dxdt=\int_Q L_y[x,t] z(x,t)dxdt$.  Letting $n\to\infty$  we obtain from \eqref{CriticalCon1} that
\begin{align}\label{c1}
\int_Q(L_y[x,t]z(x,t)+ L_u[x,t]v(x,t))dxdt\leq 0. 
\end{align} Hence conditions $(c_1')$  and $(c_2')$ of  $\mathcal{C}_2[(\bar y, \bar u)]$ are valid.  It remains to verify  $(c_3')$ and $(c_4')$. Let
\begin{align*}
    G(y, u)(x, t) = \bigg(G_1(y, u)(x, t), \  G_2(y, u)(x, t)\bigg),
\end{align*} where $G_1(y, u)=u$ and $G_2(y, u)= g(\cdot, \cdot, y) +\epsilon u$. Since $(y_n, u_n)\in\Phi$, we have $G_i(y_n, u_n) \in K_i$  with $n = 1, 2$ and 
\begin{align*}
    &K_1=\{v\in L^2(Q): a\leq  v(x,t)\leq b\quad \text{a.a.}\quad (x, t)\in Q\},\\
    &K_2=\{v\in L^2(Q): v(x, t)\leq 0\quad {\rm a.a.}\ (x,t)\in Q\}.
\end{align*}
By the Taylor expansion and the definitions of $z_n, v_n$, we have
\begin{align*}
    G_{iy}(\bar z)z_n + G_{iu}(\bar z)v_n  + \frac{o(t_n)}{t_n}& \in \frac{1}{t_n} \big(K_i - G_i(\bar y, \bar u) \big) 
    \subseteq {\rm cone}(K_i - G_i(\bar y, \bar u))  
    \subseteq T_{L^2(Q)}(K_i; G_i(\bar y, \bar u))
\end{align*}
where $T_{L^2(Q)}(K_i; G_i(\bar y, \bar u))$ is the tangent cone to $K_i$ at $G_i(\bar y, \bar u)$ in $L^2(Q)$. Since $T_{L^2}(K_i; G_i(\bar y, \bar u))$ is a closed convex subset of $L^2(Q)$, it is a weakly closed set of $L^2(Q)$ and so $T_{L^2}(K_i; G_i(\bar y, \bar u))$ is sequentially weakly closed. By a simple argument, we can show that $G_{iy}(\bar z)z_n + G_{iu}(\bar z)v_n    \rightharpoonup \   G_{iy}(\bar z)z + G_{iu}(\bar z)v $ in $L^2(Q)$. Letting $n\to\infty$, we obtain from the above that 
\begin{align}
    G_{iy}(\bar z)z + G_{iu}(\bar z)v \in T_{L^2}(K_i; G_i(\bar y, \bar u)).
\end{align} By Lemma 2.4 in \cite{Kien-2017}, we have 
$$
T_{L^2}(K_1; G_1(\bar y, \bar u))=\{ v\in L^2(Q): v(x, t)\in T([a, b]; \bar u(x,t))\quad {\rm a.a.}\quad (x, t)\in Q\}.
$$ and 
$$
T_{L^2}(K_2; G_2(\bar y, \bar u))=\{ v\in L^2(Q): v(x, t)\in T((-\infty, 0]; G_2[x,t])\quad {\rm a.a.}\ (x, t)\in Q\}. 
$$ It follows that 
$$
v(x,t)\in T([a, b]; \bar u(x, t))\quad \text{a.a.}\quad (x, t)\in Q
$$ and 
$$
g_y[x, t] z(x,t) + \epsilon v(x,t)\in T((-\infty, 0]; g[x,t]+\epsilon \bar u(x, t))\quad \text{a.a.}\quad (x, t)\in Q
$$ Hence $(z, v)$ satisfies  $(c_3')$ and $(c_4')$. Consequently, $(z,v)\in\mathcal{C}_2[(\bar y, \bar u)]$ and the claim is justified.  We now show that $(z, v)=0$.  
 Indeed, it follows from the assertion $(iii)$ of Theorem {\ref{Theorem 1}} that 
 \begin{align*}
     \int_Q e(g(y_n) + \epsilon u_n - g(\bar y) - \epsilon \bar u) dxdt \le  0 \quad \text{and} \quad \int_Q \widehat e (u_n - \bar u) dxdt \leq 0.
 \end{align*} This and definition of $\mathcal{L}$ in (\ref{Lagrange.function.2}), yield 
 \begin{align*}
    & \mathcal L (y_n, u_n, \varphi, e, \widehat e)  -  \mathcal L (\bar y, \bar u, \varphi, e, \widehat e)\\
     &\le J(y_n, u_n)-J(\bar y, \bar u) \nonumber + \int_Q\varphi\big(\frac{\partial (y_n-y)}{\partial t} +A(y_n-\bar y) + f(y_n)-f(\bar y)- (u_n-\bar u)\big)  dxdt  \nonumber \\ 
     &+  \int_{\Omega} \varphi(x, 0)[y_n(x, 0) - \bar y(x, 0)]  dx.
 \end{align*}
 Combining this with (\ref{Zn}) and (\ref{KeyIq2}), we obtain 
 \begin{align}
     \mathcal L (y_n, u_n, \varphi, e, \widehat e)  -  \mathcal L (\bar y, \bar u, \varphi, e, \widehat e)  \le J(y_n, u_n)-J(\bar y, \bar u) \le o(t_n^2).
 \end{align}
 Using a Taylor expansion for $\mathcal{L}$ and noting that $\nabla_{(y,u)}\mathcal{L}(\bar y, \bar u, \varphi, e, \widehat e)=0$, we get 
$$
\frac{t_n^2}{2}\nabla^2_{(y,u)}\mathcal{L}(\bar y, \bar u, \varphi, e, \widehat e)[(z_n, v_n), (z_n, v_n)] +o(t_n^2)\leq o(t_n^2),
$$ or, equivalently,
\begin{equation}\label{IneqL1}
\nabla^2_{(y,u)}\mathcal{L}(\bar y, \bar u, \varphi, e, \widehat e)[(z_n, v_n), (z_n, v_n)] \leq \frac{o(t_n^2)}{t_n^2}.
\end{equation} 
This means
\begin{align}\label{IneqL2}
\int_Q (L_{yy}[x,t]z_n^2 + 2L_{yu}[x,t]z_n v_n + eg_{yy}[x,t]z_n^2  +\varphi f''(\bar y) z_n^2)dxdt 
+\int_Q L_{uu}[x,t]  v_n^2 dxdt \leq \frac{o(t_n^2)}{t_n^2}.
\end{align}
Since $L_{uu}[x,t]  \ge \Lambda>0$, the functional 
\begin{align*}
    \xi \mapsto  \int_Q L_{uu}[x,t]  \xi^2 dxdt
\end{align*}
is convex and so it is sequentially lower semicontinous. Hence 
$$
\lim_{n\to\infty} \int_Q L_{uu}[x,t]  v_n^2 dxdt\geq \int_Q L_{uu}[x,t]  v^2 dxdt. 
$$ 
 By letting $n\to\infty$ and using the fact that $z_n\to z$ in $L^2(Q)$, we obtain from \eqref{IneqL2} that 
$$
 \nabla^2_{(y,u)}\mathcal{L}(\bar y,\bar u, \varphi, e, \widehat e)[(z, v), (z, v)] \leq 0.
$$ Combining this with \eqref{StrictSOSCond}, we conclude that $(z, v)=(0, 0)$.

Finally, from \eqref{IneqL2} and $\|v_n\|_{L^2(Q)}=1$, we have 
\begin{align*}
&\int_Q \big(L_{yy}[x,t]z_n^2 + 2L_{yu}[x,t]z_n v_n +\varphi f''(\bar y) z_n^2\big)dxdt  +  \int_Qeg_{yy}[x,t]z_n^2 dxdt  +\Lambda \\
&\le \int_Q \big(L_{yy}[x,t]z_n^2 + 2L_{yu}[x,t]z_n v_n +\varphi f''(\bar y) z_n^2\big)dxdt  +  \int_Qeg_{yy}[x,t]z_n^2 dxdt  +\int_Q L_{uu}[x,t]  v_n^2 dxdt\\
&\leq \frac{o(t_n^2)}{t_n^2}.
\end{align*} 
By letting $n\to\infty$ and using $(z, v)=(0, 0)$, we get $\Lambda\leq 0$ which is impossible. Therefore, the proof of  theorem is complete.   
\end{proof}

\section{H\"{o}lder continuity of multipliers and optimal solutions}

In this section we show that under strengthen hypotheses on $y_0$, $L, g$ and $\partial\Omega$, then the multipliers and optimal solution are H\"{o}lder continuous. We require the following assumptions. 

\noindent $(H3')$ $L$ is a continuous  function and  for each $(x,t)\in \Omega\times [0, T]$, $L(x, t, \cdot, \cdot)$ is of class $C^2$ and satisfies the following property:  for each $M>0$, there exists $k_{L,M}>0$ such that 
\begin{align*}
    &|L(x_1, t_1, y_1, u_1)-L(x_2, t_2, y_2, u_2)|+|L_y(x_1, t_1, y_1, u_1)-L_y(x_2, t_2, y_2, u_2)|\\
    &+|L_u(x_1, t_1, y_1, u_1)- L_u(x_2, t_2, y_2, u_2)| \leq k_{L, M} (|x_1-x_2|+ |t_1-t_2|+ |y_1-y_2|+ |u_1-u_2|),\\
    &{\rm and}\\ 
    &|L_{yy}(x, t, y_1, u_1)-L_{yy}(x, t, y_2, u_2)|+|L_{yu}(x, t, y_1, u_1)-L_{yu}(x, t, y_2, u_2)|\\
    &+ |L_{uu}(x, t, y_1, u_1)-L_{uu}(x, t, y_2, u_2)|\leq k_{L,M}(|y_1-y_2|+ |u_1-u_2|)
\end{align*}  for all $(x, t), (x_i, t_i)\in Q$ and $(y_i, u_i)\in \mathbb{R}\times\mathbb{R}$ satisfying $|y_i|, |u_i|\leq M$ with $i=1,2$.  $g$ is a continuous  function and  for each $(x,t)\in \Omega\times [0, T]$, $g(x, t, \cdot)$ is of class $C^2$ and satisfies the following property:  for each $M>0$, there exists $k_{g,M}>0$ such that 
\begin{align*}
    &|g(x_1, t_1, y_1)-g(x_2, t_2, y_2)|+|g_y(x_1, t_1, y_1)-g_y(x_2, t_2, y_2)|\\
    &\quad \quad \quad \quad +|g_u(x_1, t_1, y_1)- g_u(x_2, t_2, y_2)| \leq k_{g, M} (|x_1-x_2|+ |t_1-t_2|+ |y_1-y_2|),\\
    &{\rm and}\\ 
    &|g_{yy}(x, t, y_1)-g_{yy}(x, t, y_2)|+|g_{yu}(x, t, y_1)-g_{yu}(x, t, y_2)|\\
    &\quad \quad \quad \quad + |g_{uu}(x, t, y_1)-g_{uu}(x, t, y_2)|\leq k_{g,M}|y_1-y_2|
\end{align*}  for all $(x, t), (x_i, t_i)\in Q$ and $y_i\in \mathbb{R}$ satisfying $|y_i|\leq M$ with $i=1,2$.

\noindent $(H4')$   $f_y[x, t]\geq -\frac{1}{\epsilon} g_y[x, t]\geq 0$ for a.a. $(x, t) \in Q$  and there exists a number $\gamma>0$ such that
\begin{align}
\label{dk.tach.1}
    a -\bar u(x, t) +  g[x, t] + \epsilon\bar u(x,t) \leq - \gamma \quad {\rm a.a.} \  (x, t) \in Q.
\end{align}
 Moreover, the following condition is satisfied  
\begin{align}\label{dk.tach.manh}
    g[x, t] +\epsilon b < 0 \quad {\rm on} \quad \{ (x, t) \in Q: \bar u(x, t) = b \}.
\end{align}

\noindent $(H5')$ The boundary $\partial\Omega$ satisfies the property of positive geometric density, that is, there exist number $\alpha^*\in (0, 1)$ and $R_0>0$ such that for all $x_0\in\partial \Omega$ and all $R\leq R_0$, one has
$$
|\Omega\cap B(x_0, R)|\leq (1-\alpha^*) |B(x_0, R)|,
$$ where $|C|$ denote Lebesgue measure of a measurable set $C$ in $\mathbb{R}^N$.

\noindent $(H6')$ There exists a number $\gamma_1 > 0$ such that 
\begin{align}
    L_{uu}(x, t, y, u) \ge \gamma_1  \quad \quad  \forall (x, t, y, u) \in Q \times \mathbb R \times \mathbb R.
\end{align}

    Hypothesis $(H5')$ is  important for the H\"{o}lder continuity of solutions to the semilinear parabolic equations (see \cite{E.Di.Benedetto}).  By using the theorem of supporting hyperplane \cite[Theorem 1.5, p.19]{H.Tuy}, we can show that if $\Omega \subset \mathbb R^N$ is a nonempty convex set, then  $\partial\Omega$ satisfies $(H5')$ with $\alpha^*=1/2$ and $R_0>0$.

The following theorem gives regularity of $(\bar y, \bar u)$ and multipliers. 

\begin{theorem} \label{Theorem-Regularity} Suppose $y_0\in H^2(\Omega)\cap H_0^1(\Omega)$, hypotheses $(H1), H(2), (H3')-(H6')$, and $(\bar y, \bar u) \in \Phi$ is a locally optimal solution of problem \eqref{P1}-\eqref{P5}. Then  there exist  multipliers $\varphi, e$ and $\widehat e$  which satisfy conditions $(i)-(iv)$ of Theorem \ref{Theorem 1} and  $\bar y, \bar u, \varphi, e$ and $\widehat e$ are H\"{o}lder continuous on $\bar Q$.
\end{theorem}
\begin{proof} By the Sobolev inequality (see \cite[Theorem 6, p. 270]{Evan}) we have
$y_0\in C^{0, \alpha}(\Omega)\cap H_0^1(\Omega)$ for some $\alpha\in (0, 1]$.  Since implications $(H3')\Rightarrow (H3)$ and $(H4')\Rightarrow (H4)$ are valid, the existence of multipliers $\varphi$ and $e$ follows from Theorem \ref{Theorem 1}.   It remains to show that $\bar y, \bar u, \varphi$ and $e$ are H\"{o}lder continuous.  

\noindent \textbf{Step 1.}  Showing that $\bar y$ and $\varphi$ are H\"{o}lder continuous functions.\\
    Since $\bar y \in L^\infty(Q)$, $\bar y$ is as a bounded solution of the parabolic (1.2)-(1.3). From this, $y_0\in C^{0, \alpha}(\bar \Omega)$ and the assumption $(H5')$, according to \cite[Corollary 0.1]{E.Di.Benedetto},  $\bar y$ is $\alpha-$H\"{o}lder continuous on $\bar Q$ with exponent $\alpha \in (0,1)$. By setting $\widehat \varphi(t) = \varphi(T - t)$, then $\widehat \varphi$ is unique solution of equation
    \begin{align}
        \frac{\partial \widehat \varphi (t)}{\partial t} + A^* \widehat \varphi(t) + f'(\bar y (x, T - t)) \widehat \varphi (t) = - L_y[x, T-t] - e(x, T-t)g_y[x, T-t], \quad  \widehat \varphi (., 0) = 0.
    \end{align}
    By the same argument, we have from \cite[Corollary 0.1]{E.Di.Benedetto} that $\widehat \varphi$ is also $\alpha-$H\"{o}lder continuous on $\bar Q$. It follows that   $\varphi$ is $\alpha-$H\"{o}lder continuous.

\noindent \textbf{Step 2.} Showing that $e$ and $\widehat e$ are H\"{o}lder continuous. 

Let $Q_a, Q_b$ and $Q_0$ be defined by \eqref{Qa}, \eqref{Qb} and \eqref{Q0}, respectively.  Since $e(x, t)=\widehat e(x,t)=0$ on $Q\setminus (Q_a\cup Q_b\cup Q_0)$, we only  prove that $e$ and $\widehat e$ are H\"{o}lder continuous on $Q_a, Q_b$ and $Q_0$.  

\noindent $\bullet$  On $Q_a$ we have $\bar u(x, t) = a$. By \eqref{dk.tach.1}, we have $g[x,t] +\epsilon\bar u(x, t)  \le - \gamma <0$. Hence  $e(x, t)=0$ and so $e$ is H\"{o}lder continuous on $Q_a$.  From $(ii)$ of Theorem \ref{Theorem 1} we have
\begin{align}
    \widehat e= \varphi - L_u[., .].
\end{align}
By Step 1,  $\varphi$ is H\"{o}lder continuous function on $Q$. Let us claim that $L_u[., .]$ is H\"{o}lder continuous function on $Q_a$. Indeed, by $(H3')$, for any $(x_1, t_1), (x_2, t_2) \in Q_a$,   we have
\begin{align}
    &|L_u(x_1, t_1, \bar y(x_1, t_1), \bar u(x_1, t_1)) - L_u(x_2, t_2, \bar y(x_2, t_2), \bar u(x_2, t_2)) |  \nonumber \\
    &= |L_u(x_1, t_1, \bar y(x_1, t_1), a) - L_u(x_2, t_2, \bar y(x_2, t_2), a) | \nonumber \\
    &\le k_{L, M}(|x_1 - x_2| + | t_1 - t_2| + | \bar y(x_1, t_1) - \bar y(x_2, t_2)|). \label{L.e1}
\end{align} Therefore, $L_u[., .]$ is H\"{o}lder continuous function on $Q_a$ because
 $\bar y$ is H\"{o}lder continuous function on $Q$. The claim is justified. Consequently, 
  Thus $\widehat e$ is H\"{o}lder continuous function on $Q_a$. 

\noindent $\bullet$  On $Q_b$ we have  $\bar u(x, t) = b$. By \eqref{dk.tach.manh},  we have $g[x, t]+\epsilon \bar u(x, t)<0$ on $Q_b$. Hence  $e(x, t)=0$ on $Q_b$ and so  $e$ is H\"{o}lder continuous on $Q_b$.  Again, we have from $(ii)$ of Theorem \ref{Theorem 1} that $\widehat e (x, t) = \varphi(x,t) - L_u[x, t]$  on $Q_b$.  This implies that $\widehat e$ is  H\"{o}lder continuous function on $Q_b$.

\noindent $\bullet$  On $Q_0$ we have $\epsilon\bar u(x, t) + g[x, t]=0$. This and \eqref{dk.tach.1} imply that $\bar u(x, t) >a$. By \eqref{dk.tach.manh}, we get $\bar u(x, t)<b$. Hence $\widehat e(x, t)=0$ on $Q_0$ and so it is H\"{o}lder continuous on $Q_0$. Again, from $(ii)$ of Theorem \ref{Theorem 1} we have
\begin{align}
    e(x,t) = \frac{1}{\epsilon} (\varphi(x, t) -  L_u[x, t]).
\end{align} It follows that $e$ is H\"{o}lder continuous on $Q_0$.

\noindent \textbf{Step 3.}  Showing that $\bar u $ is H\"{o}lder continuous on $\bar Q$. 

Put $\Psi(x, t) := \epsilon e(x, t) + \widehat e (x, t), \ (x, t) \in Q$, then  from the step 2, $\Psi \in C^{0, \alpha}(\bar Q)$.     Fixing any $(x_i, t_i)\in Q$ with $i=1,2$ and using a Taylor' expansion, we have
      \begin{align*}
      &|L_u(x_1, t_1, \bar y(x_1, t_1), \bar u(x_1, t_1))-L_u(x_1, t_1, \bar y(x_1, t_1), \bar u(x_2, t_2)|\\
      &=|L_{uu}(x_1, t_1, \bar y(x_1, t_1), \bar u(x_2, t_2) +\theta(\bar u(x_1, t_1)-\bar u(x_2, t_2))||(\bar u(x_1, t_1)-\bar u(x_2, t_2))|\\
      &\geq \gamma_1 |(\bar u(x_1, t_1)-\bar u(x_2, t_2))|,
      \end{align*}
where $\theta\in [0,1]$ and the last inequality follows from $(H6')$. Combining this with equality $L_u[\cdot, \cdot] = \varphi - \Psi$, we get 
\begin{align}
\gamma_1 |(\bar u(x_1, t_1)-\bar u(x_2, t_2))|&\leq |L_u(x_1, t_1, \bar y(x_1, t_1), \bar u(x_1, t_1))-L_u(x_1, t_1, \bar y(x_1, t_1), \bar u(x_2, t_2))| \nonumber \\
&\leq |L_u(x_1, t_1, \bar y(x_1, t_1), \bar u(x_1, t_1))-L_u(x_2, t_2, \bar y(x_2, t_2), \bar u(x_2, t_2)| \nonumber \\
&+ |L_u(x_2, t_2, \bar y(x_2, t_2), \bar u(x_2, t_2)-L_u(x_1, t_1, \bar y(x_1, t_1), \bar u(x_2, t_2))| \nonumber  \\
&=|\varphi(x_1, t_1)-\varphi(x_2, t_2) -\Psi(x_1, t_1) + \Psi(x_2,t_2)| \nonumber  \\
&+|L_u(x_2, t_2, \bar y(x_2, t_2), \bar u(x_2, t_2)-L_u(x_1, t_1, \bar y(x_1, t_1), \bar u(x_2, t_2))| \nonumber  \\
&\leq |\varphi(x_1, t_1)-\varphi(x_2, t_2)| +|\Psi(x_1, t_1) - \Psi(x_2,t_2)| \nonumber  \\
&+  k_{L, M}(|x_1 - x_2| + | t_1 - t_2| + | \bar y(x_1, t_1) - \bar y(x_2, t_2)|). \label{L.baru}
\end{align} 
Since  $\varphi, \Psi$ and $\bar y$ are H\"{o}lder continuous function on $Q$, the estimation (\ref{L.baru}) shows that $\bar u$ is H\"{o}lder continuous on $\bar Q$. The proof of the theorem is complete.    
\end{proof}

\medskip

\noindent {\bf Acknowledgment} This research was supported by International Centre for Research and Postgraduate Training in Mathematics, Institute of Mathematics, VAST, under grant number ICRTM02-2022.01, and by Vietnam National Foundation for Science and Technology Development (NAFOSTED) under grant number 101.01-2019.308.

\end{document}